\documentclass[preprint,authoryear,9pt,a4wide]{elsarticle}
\usepackage{hyperref,aliascnt}
\usepackage{latexsym,euscript}
\usepackage{amsbsy,amscd,amsfonts,amsmath,amsopn,amssymb,amstext,amsthm,amsxtra,enumerate,bbm}

\usepackage{epsf,epsfig,graphics, color}
\usepackage{ float, xargs, pdfsync}
\usepackage[latin1]{inputenc}

\usepackage{fancybox,ifthen,twoopt}
\usepackage[textwidth=4cm, textsize=footnotesize]{todonotes}
\newtheorem{theorem}{Theorem}
\newaliascnt{proposition}{theorem}
\newtheorem{proposition}[proposition]{Proposition}
\aliascntresetthe{proposition}

\newaliascnt{lemma}{theorem}
\newtheorem{lemma}[lemma]{Lemma}
\aliascntresetthe{lemma}

\newaliascnt{example}{theorem}
\newtheorem{example}[example]{Example}
\aliascntresetthe{example}

\newaliascnt{remark}{theorem}
\newtheorem{remark}[remark]{Remark}
\aliascntresetthe{remark}

\newaliascnt{corollary}{theorem}

\aliascntresetthe{corollary}

\newaliascnt{definition}{theorem}
\newtheorem{definition}[definition]{Definition}
\aliascntresetthe{definition}

\newcounter{hypA'}

\newenvironment{hyp}[1]{\begin{sf}\refstepcounter{hyp#1}\begin{itemize}
  \item[({\bf #1\arabic{hyp#1}})]}{\end{itemize}\end{sf}}

\newenvironment{hyp*}[1]{\begin{sf}\refstepcounter{hyp#1}\begin{itemize}
  \item[({\bf #1\arabic{hyp#1}})]}{\end{itemize}\end{sf}}

\newcommand{\refhyp}[2][A]{$\mathbf{(#1\ref{#2})}$}
\newcommand{\refhyps}[3][A]{$\mathbf{(#1\ref{#2}\-\ref{#3})}$}

\renewcommand{\-}{\mbox{-}}
\newenvironment{subproof}[1]
{\textbf{\em #1}.}{\hfill $\blacktriangleleft$ \newline \indent}

\def\rme{\mathrm{e}}

\def\Xset{\mathsf{X}}
\def\Yset{\mathsf{Y}}
\def\Eset{\mathsf{E}}

\def\Xsigma{\mathcal{X}}
\def\Ysigma{\mathcal{Y}}

\def\Esigma{\mathcal{E}}

\def\rset{\ensuremath{\mathbb{R}}}
\def\nset{\ensuremath{\mathbb{N}}}
\def\zset{\ensuremath{\mathbb{Z}}}

\def\mcf{\mathcal{F}}
\def\mcm{\mathcal{M}}
\def\mcp{\mathcal{P}}
\def\mcn{\mathcal{N}}

\def\eqsp{\;}
\def\PE{\mathbb{E}}

\def\PP{\mathbb{P}}

\def\rmd{\mathrm{d}}
\newcommandx\LogInt[5][1=\theta,4=,5=Y]{\upsilon_{#4}^{#1}\langle {#5}_{#2:#3} \rangle}
\newcommand{\ball}[2]{\mathsf{B}(#1, #2)}
\newcommand{\wrt}{with respect to}
\newcommand{\rhs}{right-hand side}

\newcommand{\as}{\mbox{-a.s.}}

\newcommand{\ie}{i.e.}

\newcommandx{\aslim}[1]{\ensuremath{\stackrel{#1-\text{a.s.}}{\longrightarrow}}}

\newcommandx\sequence[3][2=n,3=\zset]{\ensuremath{\{ #1_{#2} \}_{#2 \in #3}}}

\newcommand{\CPE}[3][]
{\ifthenelse{\equal{#1}{}}%
{\mathbb{E}\left[\left. #2 \, \right| #3 \right]}
{\mathbb{E}_{#1}\left[\left. #2 \, \right| #3 \right]}
}
\newcommand{\CPEu}[3][]
{\ifthenelse{\equal{#1}{}}%
{\mathbb{E}\left[\left. #2 \, \right| #3 \right]}
{\mathbb{E}^{#1}\left[\left. #2 \, \right| #3 \right]}
}

\newcommand{\CPEv}[3][]
{\ifthenelse{\equal{#1}{}}%
{\mathbb{E}_\star\left[\left. #2 \, \right| #3 \right]}
{\mathbb{E}_\star_{#1}\left[\left. #2 \, \right| #3 \right]}
}

\newcommand{\CPP}[3][]
{\ifthenelse{\equal{#1}{}}%
{\mathbb{P}\left[\left. #2 \, \right| #3 \right]}
{\mathbb{P}_{#1}\left[\left. #2 \, \right| #3 \right]}
}

\newcommand{\CPPu}[3][]
{\ifthenelse{\equal{#1}{}}%
{\mathbb{P}\left[\left. #2 \, \right| #3 \right]}
{\mathbb{P}^{#1}\left[\left. #2 \, \right| #3 \right]}
}

\newcommandx{\chunk}[3]%
{\ensuremath{#1}_{#2:#3}}

\def\1{\mathbbm{1}}

\newcommandx\proj[2][1=,2=]{
\ifthenelse{\equal{#1}{}}
{\operatorname{X}}
{\operatorname{X}_{#1:#2}}
}

\newcommand{\eqdef}{:=}
\newcommand{\set}[2]{\left\{#1\, : \, #2\right\}}
\newcommand{\ensemble}[1]{\mathsf{#1}}

\newcommand{\shift}{\ensuremath{\operatorname{S}}}

\newcommandx\lkdM[3][1=,3=]{
\ifthenelse{\equal{#2}{}}
{ \mathsf{L}_{#1}^{#3}}
{ \mathsf{L}_{#1}^{#3}\langle #2\rangle}
}
\newcommandx\lkdMStat[3][1=,3=]{
\ifthenelse{\equal{#2}{}}
{ \bar{\mathsf{L}}_{#1}^{#3}}
{ \bar{\mathsf{L}}_{#1}^{#3}\langle #2 \rangle}
}

\newcommandx\lkd[3][1=,3=]{
\ifthenelse{\equal{#2}{}}
{ \ell_{#1}^{#3}}
{ \ell_{#1}^{#3}\langle #2\rangle}
}
\newcommandx\lkdStat[3][1=,3=]{
\ifthenelse{\equal{#2}{}}
{ \bar \ell_{#1}^{#3}}
{ \bar \ell_{#1}^{#3}\langle #2 \rangle}
}

\newcommand{\mlStat}[1]{\bar \theta_{#1}}
\newcommand{\ml}[1]{\theta_{#1}}
\newcommand{\mlY}[1]{\theta_{#1}}
\newcommand{\argmax}{\mathrm{argmax}}
\newcommand{\existLim}[1]{\to_{#1}}

\newcommandx{\normLip}[2][1=]{\mathrm{Lip}(#2;#1)}
\newcommandx{\wass}[2][1]{\lVert #2\rVert_{#1}}

\newcommandx{\wasser}[3][1=]{\mathcal{W}_{#1}\left(#2,#3\right)}
\newcommandx{\proho}[3][1=]{\mathcal{P}_{#1}\left(#2,#3\right)}

\newcommandx{\dobru}[2][1=]{\dobrush_{#1}\left( #2\right)}
\newcommand{\dobrush}{\Delta}

\newcommandx{\Pcan}[2][1=,2=]{\mathbb{P}_{#1}^{#2}}
\newcommandx{\Ecan}[2][1=,2=]{\mathbb{E}_{#1}^{#2}}
\newcommand{\Xinit}{\xi}

\newcommandx\cesp[4][1=,2=]{\ensuremath{{\mathbb E}_{#1}^{#2}\left[ \left. #3 \right| #4 \right]}}

\newcommandx{\f}[2][1=\theta]{f^{#1}\langle #2 \rangle}
\newcommand{\thv}{{\theta_\star}}

\begin{document}

%\issueinfo{5}{2012}
%\pagespan{1}{8}
%\received{October 31, 2011}
%\DOI{123-xxxxxxxxxxx}
%
\begin{frontmatter}



\title{Ergodicity of observation-driven time series models and consistency of the maximum likelihood estimator}
\author[tsp]{R. Douc}
\address[tsp]{Department  CITI, CNRS UMR 5157, Telecom Sudparis, Evry. France.}
\ead{randal.douc@telecom-sudparis.eu}

\author[uc]{P. Doukhan}
\address[uc]{Department of mathematics, University of Cergy-Pontoise, France.}
\ead{doukhan@u-cergy.fr}

\author[tpt]{E. Moulines}
\address[tpt]{Department  LTCI, CNRS UMR 5141, Telecom Paristech, Paris. France.}
\ead{moulines@telecom-paristech.fr}

\journal{Stochastic Processes and their applications}

%\subjclass{Primary 62J05; Secondary 62F03, 62F10, 62F30}
%\thanks{The research was supported by XXX}

\begin{abstract}
This paper deals with a general class of observation-driven time series models with a special focus on time series of counts. We provide conditions under which there exist strict-sense stationary and ergodic versions of such processes. The consistency of the maximum
likelihood estimators is then derived for well-specified and misspecified models.
\end{abstract}

\end{frontmatter}
There has recently been a strong renewed interest in developing models for time series of counts which arise in a wide variety of applications:  economics, finance, epidemiology, population dynamics...
Among the models proposed so far,  observation-driven models introduced by \cite{cox:1981} plays an important role (see \cite[Chapter 4]{kedem:fokianos:2002} for a comprehensive account and \cite{tjostheim:2012} for a recent survey). In time series of counts, the observations are the realisations of some integer-valued distribution (e.g. Poisson, negative binomial, ...) depending on some parameters that drives the dynamic of the model. In this paper, we focus on the so-called observation-driven time series models in which the parameter depends solely the past observations. Examples of such models include Poisson integer-valued GARCH (INGARCH) (see \cite{ferland:latour:oraichi:2006} or \cite{zhu:2012}, \cite{fokianos:rahbek:tjostheim:2009}), Poisson threshold models (see \cite{henderson:matteson:woodard:2011}), log-linear Poisson autoregression (see \cite{fokianos:tjostheim:2011}); see also  \cite{davis:dunsmuir:streett:2003}, \cite{davis:liu:2012} and \cite{neuman:2011}  for other observation-driven  models for Poisson counts.

This paper discuses the theory and inference for a general class of observation-driven models which includes the models introduced above as particular examples. Compared to the approach introduced in \cite{fokianos:tjostheim:2011}, our argument is not based on the so-called perturbation technique. Recall that this technique consists in two steps: in a first step, a perturbed version of the process is shown to be geometrically ergodic, in a second step, the perturbed process is shown to converge to the original one by letting the perturbation goes to $0$. These two steps make it possible to develop a likelihood theory on the perturbed process and then to take the limit. As argued by \cite{doukhan:2012}, this approximation technique might seem unnatural and is technically involved. In addition, it heavily  relies on the Poisson assumption. The approach developed by \cite{neuman:2011} is more direct but is based on a contraction assumption on the intensity of the Poisson variable which is not satisfied, for example neither in the log-linear Poisson autoregression model nor in  the Poisson threshold model. We do not follow the weak dependence approach which as outlined in \cite{doukhan:fokianos:tjostheim:2012} also implies unnecessary Lipshitz assumptions of the model and does not yield directly a theory for likelihood inference. Those authors apply  \cite{doukhan:wintenberger:2008} results; the latter use a contraction argument also adapted to deal with more general infinite memory models which essentially extends on assumptions \eqref{eq:contrac} below relative to the current Markov case. Those authors also derived weak dependence conditions for such models; we should anyway quote that such Taylor-made dependence conditions do not allow as performing results as the present techniques.

Our approach is based on the theory of Markov chains without irreducibility assumption. We first prove the existence of a stationary distribution using the result of \cite{tweedie:1988}. The main difficulty when the Markov chain is not necessarily irreducible consists in proving the uniqueness of the stationary distribution. For that purpose, we extend the delicate argument introduced by \cite{henderson:matteson:woodard:2011} and based on the theory of asymptotically strong Feller Markov chains (see \cite{hairer:mattingly:2006}).  Our extension introduces a drift term which adds considerable flexibility on the model assumptions and allow to cover the log-linear Poisson autoregression model under assumptions which are weaker than those reported in \cite{fokianos:tjostheim:2011}. We then establish ergodicity for the two-sided stationary version of the process under the sole assumption of existence and uniqueness of the stationary distribution. Finally, we develop the theory of likelihood inference by approximating the conditional likelihood by an appropriately defined stationary version of it, which is shown to converge using classical ergodic theory arguments. Our likelihood inference theory covers both well-specified and misspecified models. We focus on the consistency of the conditional likelihood estimator but the asymptotic normality can also be covered using stationary martingale arguments. Due to space constraints, this will be reported in a forthcoming paper.

The organization of the paper is as follows. \autoref{sec:exist:unic} formulates the model, establishes
the existence and uniqueness of the invariant distribution and shows the ergodicity and existence of some moments for the observation process.
The maximum likelihood estimates of the parameters and the relevant asymptotic theory are then derived in \autoref{sec:consist}. Examples of threshold autoregressive and log Poisson counts are used to illustrate our findings.
The proofs are given in \autoref{sec:proof}. Finally, the Appendix contains general statements about the ergodicity of Markov chains under minimal assumptions which might be of independent interest.
\section{Ergodicity of the Observation-driven time series model} \label{sec:exist:unic}
Let $(\Xset,d)$ be a locally compact, complete
and separable metric space and denote by $\Xsigma$ the associated Borel sigma-field. Let $(\Yset,\Ysigma)$ be a measurable space, $H$ a Markov kernel from $(\Xset,\Xsigma)$ to $(\Yset,\Ysigma)$ and $(x,y) \mapsto f_y(x)$ a measurable function from $(\Xset\times \Yset, \Xsigma \otimes \Ysigma)$ to $(\Xset, \Xsigma)$.
\begin{definition} \label{defi:odts:model}
An {\em observation-driven time series model on $\nset$} is a stochastic process $\{(X_n,Y_n)\, , \, n \in \nset\}$ on $\Xset \times \Yset$ satisfying the following recursions: for all $k \in \nset$,
\begin{align}
& Y_{k+1}|\mcf_k\sim H(X_k;\cdot)\eqsp,\nonumber \\
& X_{k+1}=f_{Y_{k+1}}(X_k)\eqsp, \label{eq:def:XY}
\end{align}
where $\mcf_k=\sigma(X_\ell,\, Y_\ell\, ; \, \ell \leq k\,, \ell \in \nset)$.
Similarly, $\{(X_n,Y_n)\, , \, n \in \zset\}$ is an {\em observation-driven time series model on $\zset$} if the previous recursion holds for all $k \in\zset$ with $\mcf_k=\sigma(X_\ell,\, Y_\ell\, ; \, \ell \leq k\,, \ell \in \zset)$.
\end{definition}

Observation-driven time series models have been introduced by \cite{cox:1981} and later considered by \cite{streett:2000}, \cite{davis:dunsmuir:streett:2003}, \cite{fokianos:rahbek:tjostheim:2009}, \cite{neuman:2011} and \cite{doukhan:fokianos:tjostheim:2012}.

In an \emph{observation-driven time series} model, \sequence{Y}[n][\nset] are observed whereas \sequence{X}[n][\nset] are not observed. This model shares similarities with Hidden Markov Models, the main difference lying in the fact that given $X_0$ and $k$ successive observations $Y_0,\ldots,Y_k$, \eqref{eq:def:XY} allows to compute $X_k$. In the following, the notation $\chunk{u}{s}{t}$ stands for $(u_s,\ldots,u_t)$ for $s\leq t$.
\begin{example}\label{example:garch:defi} The GARCH(1,1) model defined by
\begin{align*}
&Y_{k+1}|\chunk{\sigma^2}{0}{k},\chunk{Y}{0}{k} \sim \mcn(0,\sigma_k^2)\eqsp,\\ &\sigma_{k+1}^2=d+a\sigma_k^2+b Y_{k+1}^2 \eqsp,
\end{align*}
where $\min(d,a,b)>0$  can be written as in \eqref{eq:def:XY} by setting $X_k=\sigma_k^2$ and $f_y(x)=d+a x +by^2$.
\end{example}

\begin{example}\label{example:Poisson:threshold:defi}
The Poisson threshold model defined by
\begin{align*}
&Y_ {k+1}|\chunk{X}{0}{k},\chunk{Y}{0}{k}\sim \mcp(X_k)\eqsp,\\
&X_{k+1}=\omega+a X_k+bY_{k+1}+(cX_k+dY_{k+1})\1\{Y_{k+1} \notin (L,U)\}\eqsp,
\end{align*}
where $\mcp(\lambda)$ is the Poisson distribution of parameter $\lambda$ and $0<L<U<\infty$ can be written  as in \eqref{eq:def:XY} by setting $f_y(x)=\omega+ax+b y+(cx+dy)\1\{y \notin (L,U)\}$.

Note that $X_n$ being the parameter of a Poisson distribution, it should be nonnegative. It is therefore  usually assumed that $X_0\geq \omega$ and $\min(\omega,a,b,a+c,b+d)>0$.
\end{example}

\begin{example}\label{example:logPoisson:defi}
The log-linear Poisson autoregression model introduced by \cite{fokianos:tjostheim:2011} and defined by
\begin{align*}
&Y_{k+1}|\chunk{X}{0}{k},\chunk{Y}{0}{k} \sim \mcp(\rme^{X_k})\eqsp,\\
&X_{k+1}=d+aX_k+b \ln(1+Y_{k+1}) \eqsp,
\end{align*}
where $\mcp(\lambda)$ is the Poisson distribution of parameter $\lambda$ can also be written  as in \eqref{eq:def:XY} by setting $f_y(x)=d+ax+b \ln(1+y)$.
\end{example}

A natural question is to find conditions under which there exists a strict-sense stationary and ergodic version of the observation process $\sequence{Y}[k][\nset]$. Note For the GARCH(1,1) model as described in \autoref{example:garch:defi}, this problem can be easily solved by exploiting known results on random coefficient autoregressive processes; see for example \cite{brandt:1986}, \cite{bougerol:picard:1992} and the references therein.

Since $\sequence{Y}[k][\nset]$ is not itself a Markov chain, a classical approach is to prove the existence of a strict-sense stationary ergodic process $\sequence{Y}[k][\nset]$  as a deterministic function of an ergodic Markov chain.  To this aim, it is worthwhile to note that  $\{((X_n,Y_n),\mcf^{X,Y})\, , \, n \in \nset\}$ is a Markov chain on $(\Xset \times \Yset,\Xsigma \otimes \Ysigma)$ with respect to its natural filtration
$$
\mcf^{X,Y}=(\mcf^{X,Y}_k\, , \, k \in \nset)\, , \quad \mbox{where} \quad \mcf^{X,Y}_\ell=\sigma((X_k,Y_k)\, , \, 1 \leq k \leq  \ell , X_0) \eqsp,
$$
and that $\{(X_n,\mcf^{X})\, , \, n \in \nset\}$ is also a Markov chain on $(\Xset,\Xsigma)$ with respect to its natural filtration
$$
\mcf^{X}=(\mcf^{X}_k\, , \, k \in \nset)\, , \quad \mbox{where} \quad \mcf^{X}_k=\sigma(X_\ell\, , \, 0\leq \ell \leq k) \eqsp.
$$
Denote now by $Q$ the Markov kernel associated to $\{X_k\,, k \in \nset\}$ defined implicitly by the recursions \eqref{eq:def:XY}.

In this section, we derive general conditions expressed in terms of $H$ and $f$ under which $\{X_k\,, k \in \nset\}$ and $\{(X_k,Y_k)\,, k \in \nset\}$ admits a unique invariant probability distribution. This is a particularly tricky task when the observation process \sequence{Y}[n][\nset] is integer-valued as in \autoref{example:Poisson:threshold:defi} and \autoref{example:logPoisson:defi}. In such case, the Markov chain \sequence{X}[n][\nset] takes value on
$$
\set{f_{y_k} \circ \dots \circ f_{y_1}(x_0)}{k \in \nset,\, (y_1,\ldots,y_k) \in \zset^k} \eqsp,
$$
which is a countable subset of $\Xset$. When starting from two different points $x_0$ and $x'_0$, the values taken by \sequence{X}[n][\nset] may belong to two disjoint countable  subsets of $\Xset$. In that case, the total variation distance between $Q^n(x_0,\cdot)$ and $Q^n(x'_0,\cdot)$ is always equal to $2$ regardless the values of $n \in \nset$ and thus does not converge to $0$. We therefore stress that the results obtained in the sequel do not assume that the Markov chain is irreducible.

\subsection{Coupling construction and main results}
The proof is based on a coupling construction on Markov chains which is now described.

Introduce a kernel $\bar H$ from $(\Xset^2,\Xsigma^{\otimes 2})$ to $(\Yset^2,\Ysigma^{\otimes 2})$ satisfying the following conditions on the marginals: for all $(x,x') \in \Xset^2$ and $A \in \Ysigma$,
\begin{equation}\label{eq:marginalConditionsH}
\bar H((x,x');A \times \Yset)=H(x,A)\,, \quad \bar H((x,x');\Yset \times A )=H(x',A) \eqsp.
\end{equation}
Let $\ensemble{C} \in \Ysigma^{\otimes 2}$ such that $\bar H((x,x');\ensemble{C}) \neq 0$ and consider the Markov chain $\{Z_k=(X_k,X'_k,U_k)\,, \,  n \in \nset\}$ on the "extended" space $(\Xset^2\times \{0,1\}, \Xsigma^{\otimes 2}\otimes \mcp(\{0,1\}))$ with transition kernel $\bar Q$ implicitly defined as follows. Given $Z_k=(x,x',u) \in \Xset^2\times\{0,1\}$, draw $(Y_{k+1},Y'_{k+1})$ according to $\bar H((x,x');\cdot)$ and set
 \begin{align*}
 & X_{k+1}= f_{Y_{k+1}}(x)\, , \quad  X'_{k+1}= f_{Y'_{k+1}}(x')\, , \quad U_{k+1}=\1_{\ensemble{C}}(Y_{k+1},Y'_{k+1})\eqsp, \\
 &Z_{k+1}=(X_{k+1},X'_{k+1},U_{k+1})\eqsp.
 \end{align*}
 The conditions on the marginals of $\bar H$, given by \eqref{eq:marginalConditionsH} also imply  conditions on the marginals of $\bar Q$: for all $A \in \Xsigma$ and  $z=(x,x',u) \in \Xset^2 \times \{0,1\}$,
\begin{equation}
\label{eq:marginalConditionsQ}
\bar Q(z;A \times \Xset \times \{0,1\})=Q(x;A)\, , \quad \bar Q(z;\Xset \times A  \times \{0,1\})=Q(x';A) \eqsp.
\end{equation}
For $z=(x,x',u) \in  \Xset^2\times \{0,1\}$, write
\begin{equation}\label{eq:expressionAlpha}
\alpha(x,x')=\bar Q(z;\Xset^2 \times \{1\})=\bar H((x,x');\ensemble{C})\neq 0\eqsp.
\end{equation}
The quantity $\alpha(x,x')$ is thus the probability of the event $\{U_{1}=1\}$ conditionally on $Z_0$, taken on $Z_0=z$. Denote by $Q^\sharp$ the kernel on $(\Xset^2,\Xsigma^{\otimes 2})$ defined by: for all $z=(x,x',u) \in \Xset^2\times \{0,1\}$ and $A \in \Xsigma^{\otimes 2}$,
$$
Q^\sharp((x,x'); A)= \frac{\bar Q(z;A \times \{1\})}{\bar Q(z;\Xset^2 \times \{1\})}\eqsp,
 $$
 so that using \eqref{eq:expressionAlpha},
\begin{equation}
\label{eq:barQ:alpha:Qstar}
\bar Q(z;A \times \{1\})=\alpha\left(x,x'\right) Q^\sharp((x,x'); A) \eqsp.
\end{equation}
This shows that $Q^\sharp((x,x');\cdot)$ is the distribution of $(X_1,X'_1)$ conditionally on $(X_0,X'_0,U_1)=(x,x',1)$. Consider the following assumptions:

\begin{hyp}{A}\label{assum:weakFeller}
The Markov kernel $Q$ is weak Feller. Moreover, there exist a compact set $C\in \Xsigma$, $(b,\epsilon) \in \rset^+_* \times \rset^+_*$ and a function $V : \Xset \to \rset^+$ such that
\begin{equation}
\label{eq:driftCond}
QV\leq V -\epsilon+b\1_C \eqsp.
\end{equation}
\end{hyp}

Following \cite[Definition 6.1.2]{meyn:tweedie:1993}, a point $x_0 \in \Xset$ is said to be {\em reachable} for the Markov kernel $Q$ if for all $x \in \Xset$ and all open sets $A$ containing $x_0$, we have $\sum_n Q^n(x,A) > 0$.
\begin{hyp}{A} \label{assum:reachable}
The Markov kernel $Q$ has  a reachable point.
\end{hyp}
In what follows, if $(\Eset,\Esigma)$ a measurable space, $\Xinit$ a probability distribution on  $(\Eset,\Esigma)$ and $R$ a Markov kernel on $(\Eset,\Esigma)$, we denote by $\PP_{\Xinit}^R$ the probability induced on $(\Eset^\nset,\Esigma^{\otimes \nset})$ by a Markov chain with transition kernel $R$ and initial distribution $\Xinit$. We denote by $\PE_{\Xinit}^R$ the associated expectation.

\begin{hyp}{A} \label{assum:asympStrgFeller:general}
There exist a kernel $\bar Q$ on $(\Xset^2\times \{0,1\}, \Xsigma^2\otimes \mcp(\{0,1\}))$, a kernel $Q^\sharp$ on $(\Xset^2, \Xsigma^{\otimes 2})$ and a measurable function $\alpha: \Xset^2 \to \{0,1\}$ satisfying  \eqref{eq:marginalConditionsQ} and \eqref{eq:barQ:alpha:Qstar}, a measurable function $W: \Xset^2 \to [1,\infty)$ and real numbers  $(D,\zeta_1,\zeta_2,\rho) \in (\rset^+)^3 \times (0,1)$   such that for all $(x,x') \in \Xset^2$,
%such that
%\begin{equation}\label{eq:defiArho}
%\ensemble{C}\subset\set{(y,y') \in \Yset^2}{\sup_{x \neq x'}\frac{d(f_y(x),f_{y'}(x'))}{d(x,x')}\leq \rho}\eqsp,
%\end{equation}
%, a measurable function $W: \Xset^2 \to \rset^+$ and real numbers $(\rho,\beta) \in (0,1)\times \rset^+$  such that for all $(x,x') \in \Xset^2$, $\bar H((x,x');\ensemble{C})\neq 0$ and
\begin{align}
&1-\alpha(x,x') \leq   d(x,x') W(x,x')\label{eq:hyp:beta:gene}\\
& \PE^{Q^{\sharp}}_{\delta_x \otimes \delta_{x'}}[d(X_n,X'_n)] \leq D \rho^n d(x,x') \label{eq:majo:d}\\
& \PE^{Q^{\sharp}}_{\delta_x \otimes \delta_{x'}}[d(X_n,X'_n)W(X_n,X'_n)] \leq D \rho^n d^{\zeta_1}(x,x')W^{\zeta_2}(x,x') \label{eq:majo:d:w}
\end{align}
Moreover, for all $x \in \Xset$, there exists $\gamma_x>0$ such that
\begin{equation}
\label{eq:definition-gamma-x}
\sup_{x' \in \ball{x}{\gamma_x}} W(x,x')<\infty\eqsp.
\end{equation}
\end{hyp}
\begin{remark}
The assumption \refhyp[A]{assum:weakFeller} implies by \cite[Theorem 2]{tweedie:1988} that the Markov kernel $Q$ admits at least one stationary distribution. Assumptions \refhyps{assum:reachable}{assum:asympStrgFeller:general} are then used to show that this stationary distribution is unique.
\end{remark}
\begin{remark}
These assumptions weaken the Lipshitz conditions obtained by \cite[eq (15)]{henderson:matteson:woodard:2011} by introducing a "drift" function $W$ in \eqref{eq:hyp:beta:gene}. This allows to treat for example the Log-linear Poisson autoregression under minimal assumptions. It thus answers to an open question raised by \cite[p. 816]{henderson:matteson:woodard:2011} on dealing with models which do not satisfy Lipshitz condition as expressed in \cite[eq (15)]{henderson:matteson:woodard:2011}.
\end{remark}
\begin{remark}
Eq \eqref{eq:barQ:alpha:Qstar} shows that we can simulate $(X_1,X'_1,U_1)$ according to $\bar Q((x,x',u);\cdot)$ as follows. Toss a coin with probability of heads $\alpha(x,x')$. If the coin lands head, then set $U_1=1$ and draw $(X_1,X'_1) \sim Q^\sharp((x,x');\cdot)$. Otherwise, set $U_1=0$ and draw $(X_1,X'_1)$ according to
\begin{equation}
\label{eq:residual}
A \mapsto  \frac{\bar Q((x,x',u);A \times \{0\})}{1 -\alpha(x,x')}\eqsp.
\end{equation}
Under \eqref{eq:majo:d} and \eqref{eq:majo:d:w}, the stochastic processes
$$
\{d(X_k,X'_k)\,, \eqsp k \in \nset \}\,, \quad \mbox{and}\quad \{d(X_k,X'_k)W(X_k,X'_k)\,, \eqsp k \in \nset \}\eqsp,
$$
conditionally on the fact that the coin lands heads repeatedly, goes geometrically fast to $0$ in expectation. When the coin lands tail, nothing is assumed about the behavior of these processes but we can bound the probability of this event by \eqref{eq:hyp:beta:gene}.
\end{remark}

\begin{theorem} \label{thm:uniqueInvariantProbaMeasure}
Assume that \refhyps[A]{assum:weakFeller}{assum:asympStrgFeller:general} hold. Then, the Markov kernel $Q$ admits a unique invariant probability measure.
\end{theorem}
\begin{proof}
The proof is postponed to \autoref{sec:proof}.
\end{proof}
Note that \autoref{thm:uniqueInvariantProbaMeasure} does not provide a rate of convergence to the stationary distribution. Nevertheless, when discussing inference in these models, some moment conditions with respect to the stationary distribution are needed. The following Lemma allows to assess if a function $f$ is integrable with respect to an invariant distribution of the Markov kernel $Q$.
\begin{lemma} \label{lem:piV}
Assume that the Markov kernel $Q$ admits an invariant kernel $\pi$ and that there exist a measurable function $V: \Xset \to \rset^+$ and real numbers $(\lambda,\beta) \in (0,1) \times \rset^+$ such that $QV \leq \lambda V +\beta$.
Then,
$$
\pi V\leq \beta/(1-\lambda) <\infty\eqsp.
$$
\end{lemma}
\begin{proof}
The proof is postponed to \autoref{sec:proof}.
\end{proof}

\begin{proposition} \label{prop:ergo-Y:general}
Assume that the Markov kernel $Q$ admits a unique invariant probability measure. Then, there exists a strict-sense stationary ergodic process on $\zset$, \sequence{Y}[n][\zset], solution to the recursion \eqref{eq:def:XY}.
\end{proposition}
\begin{proof}
Denote by $\pi$ the unique invariant distribution of the Markov kernel $Q$.
Now, let $\{(X_n,Y_n)\, , n \in \nset\}$ be the Markov chain satisfying \eqref{eq:def:XY}. If $\bar \pi$ is an invariant distribution for $\{(X_n,Y_n)\, , n \in \nset\}$, then the marginal distribution $A \mapsto \bar \pi(A \times \Yset)$ is a stationary distribution for the Markov kernel $Q$ and since $\pi$ is unique, $\bar \pi(A \times \Yset)=\pi(A)$. If $(X_0,Y_0) \sim \bar \pi$, then by \eqref{eq:def:XY}, $(X_1,Y_1)$ is distributed according to $B \mapsto \iint \pi(\rmd x) H(x;\rmd y_1) \1_B(f_{y_1}(x),y_1)$. Since $\bar \pi$ is an invariant distribution for $\{(X_n,Y_n)\, , n \in \nset\}$, we therefore obtain,
\begin{equation}
\label{eq:definition-bar-pi}
\bar \pi(B)=\iint \pi(\rmd x) H(x;\rmd y_1) \1_B(f_{y_1}(x),y_1)\eqsp, \quad \text{for all $B \in \Xsigma \otimes \Ysigma$.}
\end{equation}
Thus, the  Markov chain$\{(X_n,Y_n)\, , n \in \nset\}$ has a unique invariant distribution given by \eqref{eq:definition-bar-pi}. By applying \autoref{thm:unicInvMeasImplyErgodic} and \autoref{thm:ergod:N:Z}, there exists a strict-sense stationary ergodic process on $\zset$, $\{(X_n,Y_n)\, , n \in \zset\}$, solution to the recursion \eqref{eq:def:XY}. The proof follows.
\end{proof}
We end the section by providing some practical conditions for checking \eqref{eq:majo:d} and \eqref{eq:majo:d:w} in \refhyp[A]{assum:asympStrgFeller:general}.
\begin{lemma} \label{lem:practical:conditions}
Assume that either \eqref{item:assum:rho:beta} or \eqref{item:assum:W:bounded} or \eqref{item:assum:general} (defined below) holds.
\begin{enumerate}[(i)]
\item \label{item:assum:rho:beta} There exists $(\rho, \beta) \in (0,1) \times \rset$ such that for all $(x,x')\in \Xset^2$
\begin{align}
& d(X_1,X'_1) \leq \rho d(x,x')\, , \quad \PP^{Q^\sharp}_{\delta_x \otimes \delta_{x'}}\as \label{eq:contrac}\\
& Q^\sharp W \leq W +\beta \label{eq:driftCond:W}
\end{align}
\item \label{item:assum:W:bounded} $\eqref{eq:majo:d}$ holds and $W$ is bounded.
\item \label{item:assum:general} $\eqref{eq:majo:d}$ holds and there exists $0<\alpha<\alpha'$ and $\beta \in \rset^+$ such that for all $(x,x')\in \Xset^2$
    \begin{align*}
    &d(x,x') \leq W^\alpha(x,x')\\
    &Q^\sharp W^{1+\alpha'} \leq W^{1+\alpha'}+\beta
    \end{align*}
\end{enumerate}
Then, \eqref{eq:majo:d} and \eqref{eq:majo:d:w} hold.
\end{lemma}
\begin{remark}
Ergodicity under Lipshitz conditions have been studied in a wide literature including \cite{sunyach:1975} or \cite{diaconis:freedman:1999}, but the fact that the various contraction conditions in \autoref{lem:practical:conditions} are related to the kernel $Q^\sharp$ and not to the kernels $Q$ or $\bar Q$ make it possible to check the assumptions \eqref{eq:majo:d} and \eqref{eq:majo:d:w} quite directly.
\end{remark}
\begin{proof}
See \autoref{sec:proof}.
\end{proof}

\subsection{Examples}
%\subsubsection{The GARCH(1,1) model}
%We consider the model described in \autoref{example:garch:defi}.
\subsubsection{A Poisson threshold model}
Existence and uniqueness of the stationary distribution for the Poisson threshold model
have been already discussed in \cite{henderson:matteson:woodard:2011}. We can obtain the same results by applying \autoref{thm:uniqueInvariantProbaMeasure} provided that assumptions \refhyps[A]{assum:weakFeller}{assum:asympStrgFeller:general} hold.
Consider a Markov chain $\{X_n\, , \ n \in \nset\}$ with a transition kernel $Q$ given implicitly by the following recursive equations:
\begin{align*}
&Y_ {n+1}|\chunk{X}{0}{n},\chunk{Y}{0}{n}\sim \mcp(X_n)\eqsp,\\
&X_{n+1}=\omega+a X_n+bY_{n+1}+(cX_n+dY_{n+1})\1\{Y_{n+1} \notin (L,U)\}\eqsp,
\end{align*}
where $0<L<U<\infty$. Moreover, to keep the parameter of the Poisson distribution positive, it is assumed that $X_0\geq \omega$ and $\min(\omega,a,b,a+c,b+d)>0$. Here, we set $\Xset=\rset^+$, $d(x,x')=|x-x'|$ and
$$
f_y(x)=\omega+ax+b y+(cx+dy)\1\{y \notin (L,U)\}\eqsp.
$$
\begin{lemma}\label{lem:threshold:A3}
Assume that $a\vee (a+c)<1$, then \refhyp{assum:asympStrgFeller:general} holds.
\end{lemma}
\begin{proof}
Define implicitly $\bar Q$ as the transition kernel Markov chain $\{Z_n\, , \ n \in \nset\}$ with $Z_n=(X_n,X'_n,U_n)$ in the following way. Given $Z_n=(x,x',u)$, if $x \leq x'$, draw independently $Y_{n+1} \sim \mcp(x)$, $V_{n+1}\sim \mcp(x'-x)$ and set $Y'_{n+1}=Y_{n+1}+V_{n+1}$. Otherwise, draw independently $Y_{n+1}' \sim \mcp(x')$ and $V_{n+1}\sim \mcp(x-x')$ and set $Y_{n+1}=Y'_{n+1}+V_{n+1}$. In all cases, set
\begin{align*}
&X_{n+1}=f^\theta_{Y_{n+1}}(x)\eqsp,\\
 &X'_{n+1}=f^\theta_{Y'_{n+1}}(x')\eqsp, \\
 &U_{n+1}=\1\{Y_{n+1}=Y'_{n+1}\}=\1\{V_{n+1}=0\}\eqsp,\\
 & Z_{n+1}=(X_{n+1},X'_{n+1},U_{n+1})\eqsp.
\end{align*}
Note again that if $Y\sim \mcp(\lambda)$, $V\sim \mcp(\lambda')$ and $(Y,V)$ are independent, then $Y+V\sim \mcp(\lambda+\lambda')$. This implies that $\bar Q$ satisfies the marginal conditions \eqref{eq:marginalConditionsQ}.
Define for all $x^\sharp=(x,x') \in \rset^2$, $Q^\sharp(x^\sharp,\cdot)$ as the law of $(X_1,X'_1)$ where
\begin{equation*}
X_{1}=f^\theta_Y(x)\, , \quad X'_{1}=f^\theta_Y(x')\eqsp,
\end{equation*}
and $Y \sim \mcp(x \wedge x')$, and set, for all $x^\sharp=(x, x') \in \rset^2$,
$$
\alpha(x^\sharp)=\exp\left\{-|x-x'|\right\}\eqsp.
$$
With these definitions, obviously, $\bar Q$, $\alpha$ and $Q^\sharp$ satisfy \eqref{eq:barQ:alpha:Qstar}.
Moreover, using $1-\rme^{-u}\leq u$, we obtain
\begin{equation*}
1-\alpha(x^\sharp)=1-\exp\left\{-|x-x'|\right\} \leq |x-x'|\eqsp.
\end{equation*}
so that \eqref{eq:hyp:beta:gene} holds with $W=\1_{\rset^2}$. To obtain \eqref{eq:majo:d} and \eqref{eq:majo:d:w}, we apply \autoref{lem:practical:conditions} by checking \eqref{item:assum:rho:beta} in \autoref{lem:practical:conditions}.
\begin{equation}
\label{eq:threshold:brutos}
\Pcan[\delta_x\otimes\delta_{x'}][Q^\sharp]\{|X_1-X'_1|= |a+c\1\{Y_1 \notin (L,U)\}| |x-x'|\leq \rho |x-x'|\}=1\eqsp,
\end{equation}
where $\rho=a\vee(a+c)<1$. The function $W$ being constant, \eqref{item:assum:rho:beta} holds and the proof is completed.
\end{proof}

\begin{proposition}\label{prop:threshold:unique}
Assume that $(a+b+c+d)\vee a<1$, then the Markov kernel $Q$ admits a unique stationary distribution $\pi$. Moreover, $\pi V<\infty$ where $V$ is the function $V:\, \rset^+\to \rset^+$ defined by $V(x)=x$.
\end{proposition}
\begin{proof}
According to \autoref{thm:uniqueInvariantProbaMeasure} and \autoref{lem:threshold:A3}, it is enough to show \refhyps{assum:weakFeller}{assum:reachable} to obtain the existence and unicity of an invariant probability measure $\pi$. We start with \refhyp{assum:weakFeller}. A random variable of distribution $\mcp(\lambda)$ converges weakly to a random variable of distribution $\mcp(\lambda')$ as $\lambda \to \lambda'$. This implies by Slutsky's Lemma that if $N: \rset_+ \to \nset$, $x \mapsto N(x)$ is a Poisson process of unit intensity, then
$$
X_1(x) = \omega+a x+bN(x)+(cx+dN(x))\1\{N(x) \notin (L,U)\}
$$
converges weakly to $X_1(x')$ as $x \to x'$. Therefore, $Q$ is weakly Feller. Moreover, it can be readily checked that the nonnegative function $V(x)=x$ ($V$ is indeed nonnegative as a function defined on $\Xset=\rset^+$) satisfies:
$$
Q V(x) = (a+b+c\PP[N(x)\notin (L,U)]+d\PE[N(x)\1_{N(x)\notin (L,U)}]/x)V(x)+\omega\eqsp.
$$
It can be easily checked that
$$
\lim_{x \to \infty} \PP[N(x)\notin (L,U)] = 1 \, , \quad \mbox{and}\quad \lim_{x \to \infty} \PE[N(x)\1_{N(x)\notin (L,U)}]/x = 1\eqsp,
 $$
so that
$$
\lim_{x \to \infty} \frac{Q V(x)}{V(x)} = a+b+c+d<1 \, , \quad \mbox{and}\quad \sup_{0\leq x\leq M} Q V(x)<\infty\, , \, \forall M \in \rset^+\eqsp.
 $$
 These two properties imply that there exist $(\lambda, \beta) \in (0,1) \times \rset^+$ such that
  \begin{equation}
  \label{eq:threshold:driftGeom}
  QV \leq \lambda V+\beta\eqsp.
  \end{equation}
  Thus, the drift condition \eqref{eq:driftCond} holds. Thus, \refhyp{assum:weakFeller} is satisfied. Set $x_\infty=\omega/(1-a-c)$ and let $C$ be an open set containing $x_\infty$.
  Let $x \in \rset$ and define recursively the sequence $x_0=x$ and for all $k \geq 1$, $x_k=\omega+(a+c) x_{k-1}$. Since $(a+c) < 1$, this sequence has a unique limiting point, $\lim_{n \to \infty} x_n=x_\infty$. Therefore, there exists some $n$ such that for all $k\geq n$, $x_k \in C$. For such $n$, we have
\begin{multline*}
Q^n(x,C)=\Pcan[\delta_x][Q](X_n \in C)\geq \Pcan[\delta_x][Q](X_n \in C,\ Y_1=\ldots=Y_{n}=0) \\
=\Pcan[\delta_x][Q](Y_1=\ldots=Y_{n}=0)>0\eqsp.
\end{multline*}
so that \refhyp{assum:reachable} holds. Moreover, since the function $V(x)=x$ satisfies \eqref{eq:threshold:driftGeom}, \autoref{lem:piV} show that $\pi V<\infty$.
\end{proof}
\begin{remark}
In the proof of \autoref{lem:threshold:A3}, we check \refhyp{assum:asympStrgFeller:general} by verifying \autoref{lem:practical:conditions}-\eqref{item:assum:rho:beta}. In some models, applying \autoref{lem:practical:conditions}-\eqref{item:assum:W:bounded} may provide more flexibility as can be seen in the following example:
\begin{align*}
&Y_ {n+1}|\chunk{X}{0}{n},\chunk{Y}{0}{n}\sim \mcp(X_n)\eqsp,\\
&X_{n+1}=\omega+(a +c\1_{Y_{n+1}=0}) X_n+bY_{n+1}+dY_{n+1}\1_{Y_{n+1}=0}\eqsp,
\end{align*}
This is a particular Poisson threshold model as defined in \autoref{example:Poisson:threshold:defi} with $(L,U)=(1/2,\infty)$. It is assumed that $X_0\geq \omega$ and $\min(\omega,a,b,a+c,b+d)>0$. Now, according to \autoref{lem:threshold:A3}, \refhyp{assum:asympStrgFeller:general} holds if $a \vee (a+c)<1$. We can now prove that \refhyp{assum:asympStrgFeller:general} holds even if $a+c>1$ provided that $a \vee (a+c\rme^{-\omega})<1$. To see this, we just adapt the proof of \autoref{lem:threshold:A3} by replacing \eqref{eq:threshold:brutos} by
$$
\PE^{Q^\sharp}_{\delta_x \otimes \delta_{x'}}(|X_1-X'_1|) =\PE^{Q^\sharp}_{\delta_x \otimes \delta_{x'}}(a+c\1\{Y_1 =0\}) |x-x'|\leq \rho|x-x'|
$$
where $\rho\eqdef a+c\rme^{-\omega}<1$. This implies that \eqref{eq:majo:d} holds so that condition \autoref{lem:practical:conditions}-\eqref{item:assum:W:bounded} holds and thus \autoref{lem:practical:conditions} concludes the proof.
\end{remark}

\subsubsection{Log-linear Poisson autoregression}
Consider a Markov chain $\{X_n\, , \ n \in \nset\}$ with a transition kernel $Q$ given implicitly by the following recursive equations:
\begin{align}
&Y_{n+1}|\chunk{X}{0}{n},\chunk{Y}{0}{n}\sim \mcp(\rme^{X_{n}}) \eqsp, \nonumber \\
&X_{n+1}=d+aX_{n}+b\ln\left(Y_{n+1} +1\right) \eqsp, \label{defi:logPoisson}
\end{align}
where $\mcp(\lambda)$ is a Poisson distribution with parameter $\lambda$. In this case, the state space is $\Xset=\rset$ which is equipped with the euclidean distance $d(x,x')=|x-x'|$ and the function $f_y$ is defined by: $f_y(x)=d+ax+b \ln(1+y)$.

\begin{lemma} \label{lem:logPoiss:asf}
If $|a+b|\vee |a| \vee |b|<1$, then \refhyp{assum:asympStrgFeller:general} holds.
\end{lemma}
\begin{proof}
Define implicitly $\bar Q$ as the transition kernel Markov chain $\{Z_n\, , \ n \in \nset\}$ with $Z_n=(X_n,X'_n,U_n)$ in the following way. Given $Z_n=(x,x',u)$, if $x \leq x'$, draw independently $Y_{n+1} \sim \mcp(\rme^{x})$ and $V_{n+1}\sim \mcp(\rme^{x'}-\rme^{x})$ and set $Y'_{n+1}=Y_{n+1}+V_{n+1}$. Otherwise, draw independently $Y'_{n+1} \sim \mcp(\rme^{x'})$ and $V_{n+1}\sim \mcp(\rme^{x}-\rme^{x'})$ and set $Y_{n+1}=Y'_{n+1}+V_{n+1}$. In all cases, set
\begin{align*}
&X_{n+1}=d+ax+b\ln\left(Y_{n+1} +1\right)\eqsp,\quad \\ &X'_{n+1}=d+ax'+b\ln\left(Y'_{n+1}+1\right)\eqsp,\quad \\
&U_{n+1}=\1\{Y_{n+1}=Y'_{n+1}\}=\1\{V_{n+1}=0\}\eqsp, \\
& Z_{n+1}=(X_{n+1},X'_{n+1},U_{n+1})\eqsp.
\end{align*}
Along the same lines as above, $\bar Q$ satisfies the marginal conditions \eqref{eq:marginalConditionsQ}.
Moreover, define for all $x^\sharp=(x,x') \in \Xset^2$, $Q^\sharp(x^\sharp,\cdot)$ as the law of $(X_1,X'_1)$ where
\begin{align}
\label{eq:definition-X-sharp-log-poisson}
&X_{1}=d+ax+b\ln\left(Y+1\right)\eqsp,  \quad Y \sim \mcp(\rme^{x \wedge x'}) \eqsp, \\
\nonumber
&X'_{1}=d+ax'+b\ln\left(Y +1\right)\eqsp,
\end{align}
 and set for all $x^\sharp=(x, x') \in \rset^2$,
$$
\alpha(x^\sharp)=\exp\left\{-\rme^{x \vee x'}+\rme^{x \wedge x'} \right\}\eqsp.
$$
With these definitions, obviously, $\bar Q$, $\alpha$ and $Q^\sharp$ satisfy \eqref{eq:barQ:alpha:Qstar}.
Using twice $1-\rme^{-u}\leq u$, we obtain
\begin{align*}
1-\alpha(x^\sharp)&=1-\exp\left\{-\rme^{x \vee x'}+\rme^{x \wedge x'} \right\}
\leq \rme^{x \vee x'}-\rme^{x \wedge x'}\\
&=\rme^{x \vee x'}(1-\rme^{-|x-x'|}) \leq  W(x,x')|x-x'|\eqsp.
\end{align*}
with $W(x^\sharp)=\rme^{|x| \vee |x'|}$ so that \eqref{eq:hyp:beta:gene} holds. To check \eqref{eq:majo:d} and \eqref{eq:majo:d:w}, we apply \autoref{lem:practical:conditions} by checking \eqref{item:assum:rho:beta} in \autoref{lem:practical:conditions}.
Note first that
$$
\Pcan[\delta_x\otimes\delta_{x'}][Q^\sharp]\{|X_1-X'_1|= |a| |x-x'|\}=1\eqsp,
$$
so that \eqref{eq:contrac} is satisfied. To check \eqref{eq:driftCond:W}, we will show that
\begin{equation}\label{eq:logPoiss:limQW/W}
\lim_{|x| \vee |x'| \to \infty} \frac{Q^\sharp W(x,x')}{W(x,x')}=0\eqsp.
\end{equation}
and for all $M>0$,
\begin{equation} \label{eq:logPoiss:bound}
\sup_{|x| \vee |x'| \leq M} Q^\sharp W(x,x')<\infty\eqsp.
\end{equation}
Without loss of generality, we assume that $x \leq x'$.
%Let $Y\sim \mcp(\rme^x)$ and define the random variables
%\begin{align*}
%&X_{1}=d+ax+b\ln\left(Y+1\right)\eqsp, \\
%&X'_{1}=d+ax'+b\ln\left(Y +1\right)\eqsp.
%\end{align*}
%With these definitions,
Using \eqref{eq:definition-X-sharp-log-poisson}, we get
\begin{equation}\label{eq:logPoiss:technicos:three}
Q^\sharp W(x,x')=\Ecan\left(\rme^{|X_1| \vee |X'_1|}\right) \leq \Ecan(\rme^{|X_1|})+\Ecan(\rme^{|X'_1|})\eqsp.
\end{equation}
First consider the second term of the \rhs\ of \eqref{eq:logPoiss:technicos:three},
\begin{equation}\label{eq:logPoiss:technicos:one}
\Ecan(\rme^{|X'_1|})\leq \rme^{|d|} \Ecan(\rme^{|a x'+b \ln(1+Y)|})\eqsp.
\end{equation}
Now, note that if $u$ and $v$ have different signs or if $v=0$, then $|u+v| \leq |u| \vee |v|$. Otherwise, $|u+v|= (u+v)\1\{v>0\}\vee (-u-v)\1\{v< 0\}$. This implies that
$$
\rme^{|u+v|} \leq \rme^{|u|}+\rme^{|v|}+ \rme^{u+v}\1\{v>0\} + \rme^{-u-v}\1\{v<0\}\eqsp.
$$
Plugging this into \eqref{eq:logPoiss:technicos:one},
\begin{multline*}
\Ecan(\rme^{|X'_1|}) \leq \rme^{|d|} \left( \rme^{|a||x'|}+ \PE[(1+Y)^{|b|}]+\rme^{ax'}\PE[(1+Y)^{b}] \1\{b >0\} \right. \\
\left. +\rme^{-ax'}\PE[(1+Y)^{-b}] \1\{b <0\}\right)\eqsp.
\end{multline*}
Note that for all $\gamma \in [0,1]$,
$$
\PE[(1+Y)^\gamma]\leq [\PE(1+Y)]^\gamma=(1+\rme^{x})^\gamma\leq 1+\rme^{\gamma x} \leq 1+\rme^{\gamma x'} \eqsp.
$$
Moreover, since $|b|\in [0,1]$, we have $b \1\{b >0\}  \in [0,1]$ and $-b  \1\{b <0\} \in [0,1]$. Therefore,
\begin{align*}
\Ecan(\rme^{|X'_1|}) &\leq \rme^{|d|} \left( \rme^{|a||x'|}+ 1+\rme^{|b| |x|}+\rme^{a x'}(1+\rme^{b x'}) \1\{b >0\} \right. \\
& \quad \quad \left. +\rme^{-ax'}(1+\rme^{-b x'}) \1\{b <0\}\right)\\
&\leq \rme^{|d|} \left( \rme^{|a||x'|}+1+\rme^{|b| |x|}+\rme^{|a| |x'|}  +\rme^{|a + b| |x'|} \right)\\
&\leq \rme^{|d|} \left( 1+4 \rme^{\gamma(|x|\vee|x'|)} \right)\eqsp,
\end{align*}
where $\gamma=|a|\vee |b|\vee |a+b|<1$. The first term of the \rhs\ of \eqref{eq:logPoiss:technicos:three} is treated as the second term by setting $x'=x$. We then have
$$
\Ecan(\rme^{|X_1|}) \leq \rme^{|d|} \left( 1+4 \rme^{\gamma(|x|\vee|x'|)} \right)\eqsp, \\
$$
so that using \eqref{eq:logPoiss:technicos:three},
\begin{equation}\label{eq:logPoiss:technicos:two}
Q^\sharp W(x,x') \leq 2 \rme^{|d|} \left( 1+4 \rme^{\gamma(|x|\vee|x'|)} \right) \eqsp.
\end{equation}
Since $\gamma \in (0,1)$ and $W(x,x')=\rme^{|x|\vee |x'|}$, \eqref{eq:logPoiss:technicos:two} implies clearly \eqref{eq:logPoiss:limQW/W} and \eqref{eq:logPoiss:bound}. The proof is completed.
\end{proof}

\begin{proposition} \label{prop:logPoiss:unicMesInv}
If $|a+b|\vee |a| \vee |b|<1$, the Markov kernel $Q$ admits a unique invariant probability measure. Moreover, $\pi V<\infty$ where $V(x)=\rme^{|x|}$.
\end{proposition}
\begin{remark}
\cite[Lemma 2.1]{fokianos:tjostheim:2011} have obtained that the Log-linear Poisson autoregression is close to a "perturbed" ergodic Log-linear Poisson process in the case where $a^2+b^2<1$ if $a$ and $b$ have different signs and $|a +b|<1$  otherwise. In both cases, we have $|a+b|<1$. In fact, if $a^2+b^2<1$, then $|a|\vee |b|<1$. Combining it with the fact that $a\wedge b \leq a+b \leq a \vee b$ when $a$ and $b$ have different signs, we obtain $|a+b|<1$. Our conditions thus extends conditions of \cite{fokianos:tjostheim:2011} and the results obtained here address an open question raised in \cite[page 566]{fokianos:tjostheim:2011}.
\end{remark}

\begin{proof}
According to \autoref{thm:uniqueInvariantProbaMeasure} and \autoref{lem:logPoiss:asf}, it is enough to show \refhyps{assum:weakFeller}{assum:reachable}. We consider first \refhyp{assum:weakFeller}. As above,
$X_1(x) = d +a x + b \ln(1+N(\rme^{x}))$
converges weakly to $X_1(x')$ as $x \to x'$. Therefore, $Q$ is weakly Feller. Moreover, following the lines of \autoref{lem:logPoiss:asf}, it can be readily checked that the function $V(x)=\rme^{|x|}$ satisfies:
$$
Q V(x) \leq \rme^{|d|} \left( 1+4 \rme^{\gamma(|x|)} \right)\eqsp,
$$
where $\gamma=|a+b|\vee |a| \vee |b|<1$. Thus,
\begin{equation}
\label{eq:logPoiss:drift}
QV(x) \leq \lambda V(x) +\beta\eqsp,
\end{equation}
 for some constants $(\lambda,\beta) \in (0,1) \times \rset^+$ showing \refhyp{assum:weakFeller}. Consider now $x=d/(1-a)$. Let $x \in \rset$ and let $C$ be an open set containing $x$. Then, by setting $x_0=x$ and for all $k \geq 1$, $x_k=d+a x_{k-1}$, we have $\lim_{n \to \infty} x_n=x$ so that there exists some $n$ such that for all $k\geq n$, $x_k \in C$. For such $n$, we have
\begin{multline*}
Q^n(x,C)=\Pcan[\delta_x][Q](X_n \in C)\geq \Pcan[\delta_x][Q](X_n \in C,\ Y_1=\ldots=Y_{n}=0) \\
=\Pcan[\delta_x][Q](Y_1=\ldots=Y_{n}=0)>0\eqsp.
\end{multline*}
so that \refhyp{assum:reachable} holds. Since \eqref{eq:logPoiss:drift} holds for the function $V(x)=\rme^{|x|}$,  \autoref{lem:piV} shows that $\pi V <\infty$. The proof follows.
\end{proof}

\section{Consistency of the Maximum Likelihood Estimator}
\label{sec:consist}
\subsection{Misspecified models}
Let $(\Theta,\mathsf{d})$ be a compact metric set of $\rset^p$, let $H$ be a Markov kernel from $(\Xset,\Xsigma)$ to $(\Yset,\Ysigma)$ and let $\{(x,y) \mapsto f^\theta_y(x)\, ,\, \theta \in \Theta\}$ be a family of measurable functions from $(\Xset\times \Yset, \Xsigma \otimes \Ysigma)$ to $(\Xset, \Xsigma)$. Assume that all $x \in \Xset$, $H(x;\cdot)$ is dominated by some $\sigma$-finite measure $\mu$ on $(\Yset,\Ysigma)$ and denote by $h(x;\cdot)$ its Radon-Nikodym derivative: $h(x;y)=\rmd H(x;\cdot)/\rmd \mu(y)$. Assume that $h(x;y)>0$ for all $(x,y) \in \Xset \times \Yset$ and that the sequence of random variables $\{(X_k,Y_k)\, ; \, k \in \nset\}$ satisfy the following recursions
\begin{align}
& Y_{k+1}| \mcf_k\sim H(X_k;\cdot) \eqsp, \nonumber \\
& X_{k+1}=f^\theta_{Y_{k+1}}(X_k) \eqsp, \label{eq:def:XY:theta}
\end{align}
where $\mcf_k$ is either  $\sigma(\chunk{X}{0}{k},\chunk{Y}{0}{k})$ or $\sigma(\chunk{X}{-\infty}{k},\chunk{Y}{-\infty}{k})$, depending whether the process is defined on $\nset$ or $\zset$.
Then,  the distribution of $(Y_1,\ldots,Y_{n})$ conditionally on $X_0=x$ has a density with respect to the product measure $\mu^{\otimes n}$ given by
\begin{equation}
\label{eq:lkd:Y}
\chunk{y}{1}{n} \mapsto \prod_{k=1}^{n} h(\f{\chunk{y}{1}{k-1}} (x);y_k) \eqsp,
\end{equation}
where we have used the convention $\f{\chunk{y}{1}{0}} (x)=x$ and the notations \begin{equation}
\label{eq:notationItere:f}
\f{\chunk{y}{s}{t}}=f^\theta_{y_t} \circ f^\theta_{y_{t-1}} \circ \dots \circ f^\theta_{y_s}\, , \quad s\leq t\eqsp.
\end{equation}

In this section, we study  the asymptotic properties of $\mlY{n,x}$, the conditional Maximum Likelihood Estimator (MLE) of the parameter $\theta$ based on the observations $(Y_1,\ldots, Y_n)$ and associated to the parametric family of likelihood functions given in \eqref{eq:lkd:Y}, that is, we consider
\begin{equation}
\label{eq:defi:mle}
\mlY{n,x} \in \argmax_{\theta \in \Theta} \lkdM[n,x]{\chunk{Y}{1}{n}}[\theta]\eqsp,
\end{equation}
where
\begin{equation}
\label{eq:defi:lkdM}
\lkdM[n,x]{\chunk{y}{1}{n}}[\theta]\eqdef   n^{-1}\ln \left( \prod_{k=1}^{n} h(\f{\chunk{y}{1}{k-1}} (x);y_k) \right)\eqsp.
\end{equation}
We are especially interested here in inference for {\em misspecified models}, that is, we {\em do not assume} that the distribution of the observations belongs to the set of distributions where the maximization occurs. In particular, $(Y_n)_{n \in \zset}$ are not necessarily the observation process associated to the recursion \eqref{eq:def:XY:theta}.

Consider the following assumptions:
\begin{hyp}{B} \label{assum:stat:ergo:Y}
$\sequence{Y}$ is a strict-sense stationary and ergodic stochastic process
\end{hyp}
Under \refhyp[B]{assum:stat:ergo:Y}, denote by $\PP_\star$ the distribution of $\sequence{Y}$ on $(\Yset^\zset,\Ysigma^{\zset})$. Write $\PE_\star$ the associated expectation.
\begin{hyp}{B} \label{assum:continuity:Y}
For all $(x,y) \in  \Xset \times \Yset$, the functions $\theta \mapsto f^\theta_y(x)$ and $v\mapsto h(v,y)$ are continuous.
\end{hyp}

\begin{hyp}{B} \label{assum:limit:Y}
There exists a family of $\PP_\star\as$ finite random variables  $$
\set{\f{\chunk{Y}{-\infty}{k}}}{ (\theta,k) \in \Theta \times \zset }
$$
 such that for all $x \in \Xset$,
\begin{enumerate}[(i)]
\item \label{item:lim:moins:infty} $\lim_{m \to \infty} \sup_{\theta \in \Theta} d(\f{\chunk{Y}{-m}{0}}(x),\f{\chunk{Y}{-\infty}{0}}) =0\,, \eqsp \PP_\star \as$,
\item  \label{item:lim:k} $\PP_\star \as$,
 $$
\lim_{k \to \infty} \sup_{\theta \in \Theta} |\ln h(\f{\chunk{Y}{1}{k-1}}(x);Y_k)- \ln h(\f{\chunk{Y}{-\infty}{k-1}};Y_k)| =0\eqsp,
$$
\item \label{item:momentCond}
$
\PE_\star \left[\sup_{\theta \in \Theta} \left(\ln h(\f{\chunk{Y}{-\infty}{k-1}};Y_k) \right)_+\right]< \infty
$
\end{enumerate}
\end{hyp}
In the following, we set for all $(\theta, k ) \in \Theta \times \nset$,
\begin{equation}
\label{eq:def:lkdStat}
\lkdStat{\chunk{Y}{-\infty}{k}}[\theta]\eqdef \ln h(\f{\chunk{Y}{-\infty}{k-1}};Y_k) \eqsp.
\end{equation}

\begin{remark}
When checking \refhyp[B]{assum:limit:Y}, we usually introduce $\f{\chunk{Y}{-\infty}{0}}$ by showing that for all $(\theta,x)\in \Theta \times \Xset$,  $\f{\chunk{Y}{-m}{-1}}(x)$ converges, $\PP_\star\as$, as $m$ goes to infinity and that the limit does not depend on $x$. We can therefore denote by $\f{\chunk{Y}{-\infty}{0}}$ this limit. With this definition, we then check  \refhyp[B]{assum:limit:Y}\eqref{item:lim:moins:infty}-\eqref{item:lim:k}-\eqref{item:momentCond}.
\end{remark}
\begin{remark} \label{rem:integerValued}
When the observation process is integer-valued, the function $y \to h(x;y)$ is a probability and thus, is less than one. It implies  that  for all $\theta \in \Theta$,
$$
\left(\lkdStat{\chunk{Y}{-\infty}{0}}[\theta] \right)_+=(\ln h(\f{\chunk{Y}{-\infty}{-1}};Y_0))_+=0\eqsp.
$$
Thus,  \refhyp[B]{assum:limit:Y}-\eqref{item:momentCond} is satisfied.
\end{remark}
Note that under \refhyp[B]{assum:continuity:Y}, $\mlY{n,x}$ is well-defined.
The following theorem establishes the consistency of the sequence of estimators $\{\mlY{n,x}\,, \, n \in \nset\}$.
\begin{theorem} \label{thm:consist:misspecified}
Assume \refhyps[B]{assum:stat:ergo:Y}{assum:limit:Y}. Then, for all $x \in \Xset$,
$$
\lim_{n \to \infty} \mathsf{d}(\mlY{n,x},\Theta_\star) =0\, , \quad \PP_\star\as
$$
where $\Theta_\star \eqdef \argmax_{\theta \in \Theta} \PE(\lkdStat{\chunk{Y}{-\infty}{0}}[\theta])$.
\end{theorem}
\begin{proof}
The proof directly follows from \autoref{thm:consist-appendix} provided that
\begin{enumerate}[(a)]
\item \label{item:one}$\PE_\star[\sup_{\theta \in \Theta} (\lkdStat{\chunk{Y}{-\infty}{0}}[\theta])_+]<\infty$,
\item \label{item:two}$\PP_\star\as$, the function $\theta \mapsto \lkdStat{\chunk{Y}{-\infty}{0}}[\theta]$ is upper-semicontinuous,
\item \label{item:three} $\lim_{n \to \infty} \sup_{\theta \in \Theta} |\lkdM[n,x]{\chunk{Y}{1}{n}}[\theta]-\lkdMStat[n]{\chunk{Y}{-\infty}{n}}[\theta]|=0$, $\PP_\star\as$ where
    $$
    \lkdMStat[n]{\chunk{Y}{-\infty}{n}}[\theta]=n^{-1}\sum_{k=1}^n \lkdStat{\chunk{Y}{-\infty}{k}}[\theta]\eqsp.
    $$
\end{enumerate}
But \eqref{item:one} follows from \eqref{item:momentCond}, \eqref{item:two} follows by combining \eqref{item:lim:moins:infty} and \refhyp[B]{assum:continuity:Y} since a uniform limit of continuous functions is continuous and  \eqref{item:three} is direct from \eqref{item:lim:k} and the definitions of $\lkdM[n,x]{\chunk{Y}{1}{n}}[\theta]$ and $\lkdMStat[n]{\chunk{Y}{-\infty}{n}}[\theta]$. The proof is completed.
\end{proof}
We end this section by providing a practical condition for checking the assumption \refhyp[B]{assum:limit:Y} when $x \mapsto f^\theta_y(x)$ is Lipshitz.
\begin{lemma} \label{lem:lipshitz}
Assume that there exists a measurable function $\varrho:\Yset \to  \rset^+$ such that for all $(\theta,y,x,x') \in \Theta \times \Yset \times \Xset^2$,
$$
d(f^\theta_y(x),f^\theta_y(x'))\leq \varrho(y) d(x,x') \eqsp.
$$
Moreover, assume that for all $x \in \Xset$,
\begin{equation*}
\text{$\PE_\star\left[\sup_{\theta \in \Theta}\ln^+ d(x,f^\theta_{Y_0}(x))\right]<\infty$,
$\PE_\star(\ln^+ \varrho(Y))<\infty$, and $\PE_\star(\ln \varrho(Y))<0$.}
\end{equation*}
Then, assumption \refhyp[B]{assum:limit:Y}-\eqref{item:lim:moins:infty} holds.
\end{lemma}
\begin{proof}
We have for all $m\geq 0$,
\begin{equation}
\label{eq:diff:x:y}
d(\f[\theta]{\chunk{Y}{-m}{0}}(x),\f[\theta]{\chunk{Y}{-m}{0}}(y)) \leq d(x,y) \prod_{\ell=0}^{m} \varrho(Y_{-\ell})
\end{equation}
Taking $y=f^\theta_{Y_{-m-1}}(x)$, we obtain
$$
d(\f[\theta]{\chunk{Y}{-m}{0}}(x),\f[\theta]{\chunk{Y}{-m-1}{0}}(x)) \leq d(x,f^\theta_{Y_{-m-1}}(x))   \prod_{\ell=0}^{m} \left[\varrho(Y_{-\ell})\right]
$$
Now, since $\PE_\thv(\ln \varrho(Y_0))<0$, $\limsup_{m \to \infty} \left( \prod_{\ell=0}^{m} \left[\varrho(Y_{-\ell})\right]\right)^{1/m}<1$
and \autoref{lem:useful} implies that
$$
\limsup_{m \to \infty} \left(  \sup_{\theta \in \Theta} d(x,f^\theta_{Y_{-m-1}}(x))
\right)^{1/m}\leq 1 .
$$
By the Cauchy root test, the series
$
\sum \sup_{\theta \in \Theta} d(\f[\theta]{\chunk{Y}{-m}{0}}(x),\f[\theta]{\chunk{Y}{-m+1}{0}}(x))
$
is convergent. This implies that
$\lim_{m \to \infty} \f[\theta]{\chunk{Y}{-m}{0}}(x)$ exists, $\PP_\star\as$ which
does not depend on $x$ by \eqref{eq:diff:x:y}.
This limit is denoted $\f[\theta]{\chunk{Y}{-\infty}{0}}$. The convergence of the series
also implies
$$
\lim_{m \to \infty} \sup_{\theta \in \Theta}d(\f[\theta]{\chunk{Y}{-m}{0}}(x),\f[\theta]{\chunk{Y}{-\infty}{0}}) = 0 \eqsp, \quad \PP_\star\as
$$
so that  \refhyp[B]{assum:limit:Y}-\eqref{item:lim:moins:infty} holds.
\end{proof}
\subsection{Well-specified models}
In this section, we focus on well-specified models, that is, \sequence{Y}[n][\nset] are assumed to be the observation process of a model defined by the recursions \eqref{eq:def:XY:theta} with $\theta=\thv \in \Theta$. In well-specified models, we stress the dependence in $\thv$ by using the notations
\begin{equation}
\label{eq:notation:well:spec}
\PP^\thv\eqdef \PP_\star\, , \quad \PE^\thv \eqdef \PE_\star\eqsp.
\end{equation}
%More generally, we denote by $\PP^\theta$ the probability induced by a stationary solution (when it exists) $\sequence{Y}[n][\zset]$ of the recursions \eqref{eq:def:XY:theta} associated to the parameter $\theta$ on $(\Yset^{\zset},\Ysigma^{\otimes \zset})$. $\PE^\theta$ denotes the associated expectation operator.
%
According to \autoref{sec:exist:unic}, to obtain \refhyp[B]{assum:stat:ergo:Y}, we only need to check that, for all $\theta \in\Theta$, \refhyps[A]{assum:weakFeller}{assum:asympStrgFeller:general} hold with $f=f^{\theta}$. If in addition, we assume that \refhyps[B]{assum:continuity:Y}{assum:limit:Y} hold and that $\Theta_\star=\{\thv\}$, then, \autoref{thm:consist:misspecified} yields: for all $x\in\Xset$,
$$
\lim_{n \to \infty} \mlY{n,x}=\thv\, , \quad \PP^\thv\as
$$
We now give conditions for having $\Theta_\star=\{\thv\}$.
\begin{proposition} \label{prop:ident}
Let $\{(X_k,Y_k)\, , \, k \in \zset\}$ be a stationary stochastic process indexed by $\zset$ which satisfies the recursions \eqref{eq:def:XY:theta} for some $\theta=\thv \in \Theta$ with $\mcf_k=\sigma(X_\ell,\, Y_\ell\, ; \, \ell \leq n\,, \ell \in \zset)$. Assume that \refhyps[B]{assum:stat:ergo:Y}{assum:limit:Y} hold and that $X_0=\f[\thv]{\chunk{Y}{-\infty}{0}}$ then, $H(X_0;\cdot)$ is the distribution of $Y_1$ conditionally on $\sigma(Y_\ell\, ; \ell \leq 0)$. If in addition,
\begin{enumerate}[(a)]
\item \label{item:ident:H} $x \mapsto H(x;\cdot)$ is one-to-one, \ie, if $H(x;\cdot) = H(x',\cdot)$, then $x=x'$,
\item \label{item:ident:f}$\f[\thv]{\chunk{Y}{-\infty}{0}}=\f[\theta]{\chunk{Y}{-\infty}{0}},\eqsp\PP^\thv\as$,  implies that $\theta=\thv$,
\end{enumerate}
then $\Theta_\star=\{\thv\}$.
\end{proposition}
\begin{remark}
Condition \autoref{prop:ident}-\eqref{item:ident:f} is similar as \cite[Assumption (A5)]{davis:liu:2012}. For the sake of clarity, we present here a self-contained proof for proving under these conditions that $\Theta_\star=\{\thv\}$.
\end{remark}
\begin{proof}
For all $A \in \Xsigma$,
\begin{multline}
\label{eq:evry2}
\CPEu[\thv]{\1_{A}(Y_1)}{\chunk{Y}{-\infty}{0}}
=\CPEu[\thv]{\CPEu[\thv]{\1_{A}(Y_1)}{X_0,\chunk{Y}{-\infty}{0}}}{\chunk{Y}{-\infty}{0}}\\
=\CPEu[\thv]{\CPEu[\thv]{\1_{A}(Y_1)}{X_0}}{\chunk{Y}{-\infty}{0}}=\CPEu[\thv]{\1_{A}(Y_1)}{X_0}=H(X_0;A)\eqsp,
\end{multline}
where we have used that $X_0$ is $\sigma(Y_\ell\, , \, \ell \leq 0)$-measurable. This concludes the first part of \autoref{prop:ident}. Now, for all $\theta \in\Theta$,
\begin{multline}
\PE^\thv\left(\ln \frac{h(\f[\thv]{\chunk{Y}{-\infty}{0}};Y_1)}{h(\f{\chunk{Y}{-\infty}{0}};Y_1)}\right)\\
=\PE^\thv\left(\CPEu[\thv]{\ln \frac{h(\f[\thv]{\chunk{Y}{-\infty}{0}};Y_1)}
{h(\f{\chunk{Y}{-\infty}{0}};Y_1)}}{\chunk{Y}{-\infty}{0}}\right)\eqsp. \label{eq:logPoiss:kulback}
\end{multline}
Under the stated assumptions,  $H(X_0;\cdot)=H(\f[\thv]{\chunk{Y}{-\infty}{0}};\cdot)$ and \eqref{eq:evry2}
shows that $H(\f[\thv]{\chunk{Y}{-\infty}{0}};\cdot)= \CPPu[\thv]{\cdot}{\chunk{Y}{-\infty}{0}}$.
Therefore, the RHS of \eqref{eq:logPoiss:kulback} is nonnegative as the expectation of a conditional Kullback-Leibler divergence. This shows that $\thv \in\Theta_\star=\argmax_{\theta \in \Theta} \PE^\thv\left(\ln h(\f[\theta]{\chunk{Y}{-\infty}{0}};Y_1)\right)$. Assume now that $\theta \in \Theta_\star$. Then, according to \eqref{eq:logPoiss:kulback},  $\PP^\thv\as$, the probability measures $H(\f[\thv]{\chunk{Y}{-\infty}{0}};\cdot)$ and $H(\f{\chunk{Y}{-\infty}{0}};\cdot)$ are equal,  so that  under \eqref{item:ident:H},
\begin{equation}
\label{eq:equal:f}
\f[\thv]{\chunk{Y}{-\infty}{0}}=\f{\chunk{Y}{-\infty}{0}} \eqsp, \quad \PP^\thv\as
\end{equation}
Under \eqref{item:ident:f}, this implies that $\theta=\thv$.
\end{proof}
\subsection{Examples}

\subsubsection{The Poisson threshold in misspecified models}
Let $K$ be a compact set of $\rset^5$ and let $\Theta$ be the following (compact) set of parameters
\begin{multline}
\label{eq:def-parameter-set}
\Theta= \left\{\theta= (\omega,a,b,c,d) \in K\, :  \right. \\
\left. \min(\omega,a,b,a+c,b+d)\geq \underline{\alpha},\,  a\vee (a+c)\leq \bar{\alpha}<1\right\} \eqsp.
\end{multline}
where $(\underline{\alpha},\bar{\alpha}) \in (0,\infty) \times (0,1)$.
 Assume that the observations $(Y_n)_{n \in \zset}$ are integer-valued and satisfy the following assumptions:
\begin{hyp}{C} \label{assum:stat:Y}
$\sequence{Y}$ is a strict-sense stationary and ergodic stochastic process
\end{hyp}
\begin{hyp}{C} \label{assum:moment:Y}
$
\PE_\star[  \ln(1 + Y_0) ] < \infty\eqsp.
$
\end{hyp}
The Poisson threshold autoregression model described in \autoref{example:Poisson:threshold:defi} may be rewritten as in \eqref{eq:def:XY:theta}, by setting $\Xset=[\underline{\alpha},\infty)$, $\Yset=\nset$ and
\begin{align}
&f^\theta_y(x)=\omega+ax+by+(cx+dy)\1\{y \notin (L,U)\} \eqsp,\label{eq:PoissThres:def:f}\\
&h(x;y)=\frac{\rmd H(x;\cdot)}{\rmd \mu}(y)=\exp(-x) x^y/{y!} \eqsp,\label{eq:PoissThres:def:h}\\
&\theta=(\omega,a,b,c,d)\eqsp, \nonumber
\end{align}
where $\mu$ is the counting measure on $\nset$. Note that
$f^\theta_y(x)=\omega + a^\theta(y) x +b^\theta(y)$
where
$a^\theta(y)=a+c\1\{y \notin (L,U)\}$ and $b^\theta(y)=b y+dy\1\{y \notin (L,U)\}$ so that for all $(\theta,y) \in \Theta \times \Yset$, $|f^\theta_y(x)-f^\theta_y(x')| \leq \bar{\alpha} |x-x'|$.
Moreover, using \eqref{eq:notationItere:f}, we have for all $s\leq t$,
\begin{equation}
\label{eq:logPoiss:itere:f}
\f{\chunk{y}{s}{t}}(x)= x\prod_ {\ell=s}^t a^\theta(y_\ell)+\sum_{j=0}^{t-s} [\omega + b^\theta(y_{t-j})] \prod_{\ell=1}^{j-1}a^\theta(y_{t-\ell}) \eqsp.
\end{equation}
With these definitions, let $\theta_{n,x}$ be the Maximum Likelihood estimator associated to the likelihood function $\lkdM[n,x]{\chunk{Y}{1}{n}}[\theta]$ as defined in  \eqref{eq:defi:mle} and \eqref{eq:defi:lkdM}.
\begin{theorem} \label{thm:PoissThres:misspecified-Consist-MLE}
Assume \refhyps[C]{assum:stat:Y}{assum:moment:Y}. Then, for all $x \in \Xset$,
$\lim_{n \to \infty} \mathsf{d}(\mlY{n,x}, \Theta_\star) = 0$, $\PP_\thv\as$
where
\begin{align}
&\Theta_\star\eqdef\argmax_{\theta \in \Theta} \PE_\star\left( \lkdStat{\chunk{Y}{-\infty}{0}}[\theta]\right)\eqsp, \label{eq:def-Theta-star-Y:PoisThres}
\end{align}
where $\lkdStat{\chunk{Y}{-\infty}{0}}[\theta]$, defined in \eqref{eq:def:lkdStat}, can be written as:
\begin{align}
&\lkdStat{\chunk{Y}{-\infty}{0}}[\theta] =Y_0 \ln(\f{\chunk{Y}{-\infty}{-1}})-\f{\chunk{Y}{-\infty}{-1}}-\ln Y_0!\\
&\f{\chunk{Y}{-\infty}{n}}= \sum_{j=0}^{\infty} [\omega + b^\theta(Y_{n-j})] \prod_{\ell=1}^{j-1}a^\theta(Y_{n-\ell})
\eqsp,\quad \forall (n,\theta) \in \zset \times \Theta\eqsp. \label{eq:def-stat-PoissThres}
\end{align}
\end{theorem}
\begin{proof}
 According to \autoref{thm:consist:misspecified}, it is sufficient to check \refhyps[B]{assum:continuity:Y}{assum:limit:Y}. \refhyp[B]{assum:continuity:Y} clearly holds.
Assumption \refhyp[C]{assum:moment:Y} allows to apply \autoref{lem:lipshitz} so that \refhyp[B]{assum:limit:Y}-\eqref{item:lim:k} holds. Using \autoref{rem:integerValued} shows that
Assumption \refhyp[B]{assum:limit:Y}-\eqref{item:momentCond} is satisfied. It remains to check \refhyp[B]{assum:limit:Y}-\eqref{item:lim:moins:infty}. By \eqref{eq:PoissThres:def:h}, for all $(x,x') \in [\underline{\alpha},\infty)^2$ and $y \in \Yset$,
\begin{align*}
|\ln h(x;y)-\ln h(x';y)| &\leq \left(y \sup_{(x,x') \in [\underline{\alpha},\infty)^2} \frac{|\ln(x)-\ln(x')|}{|x-x'|}+1\right) |x-x'|\\
& \leq \left(y/\underline{\alpha}+1\right) |x-x'|
\end{align*}
Thus,
\begin{align*}
&\sup_{\theta \in \Theta}|\ln h(\f{\chunk{Y}{1}{k-1}}(x);Y_k)- \ln h(\f{\chunk{Y}{-\infty}{k-1}};Y_k)|\\
&\quad \leq \left(Y_k/\underline{\alpha}+1\right)\sup_{\theta \in \Theta}|\f{\chunk{Y}{1}{k-1}}(x)-\f{\chunk{Y}{-\infty}{k-1}}| \\
&\quad =\left(Y_k/\underline{\alpha}+1\right) \sup_{\theta \in \Theta}\left| x\prod_ {\ell=1}^{k-1} a^\theta(Y_\ell)+\f{\chunk{Y}{-\infty}{0}}\prod_{i=0}^{k-1}a^\theta(Y_{i}) \right| \\
&\quad \leq\left(Y_k/\underline{\alpha}+1\right) \left| x+  \bar{\alpha}\f{\chunk{Y}{-\infty}{0}} \right| \bar{\alpha}^{k-1}\eqsp,
\end{align*}
which converges to $0$ as $k$ goes to infinity by applying \autoref{lem:useful} under \refhyp[C]{assum:moment:Y}.
\end{proof}
\subsection{The Poisson threshold in well-specified models}
Let $K$ be a compact set of $\rset^5$ and let $\Theta$ be the following (compact) set of parameters
\begin{multline}
\label{eq:def-parameter-set}
\Theta= \left\{\theta= (\omega,a,b,c,d) \in K\, :  \right. \\
\left. \min(\omega,a,b,a+c,b+d)\geq \underline{\alpha},\,  (a+b+c+d)\vee a\leq \bar{\alpha}<1\right\} \eqsp.
\end{multline}
where $(\underline{\alpha},\bar{\alpha}) \in (0,\infty) \times (0,1)$.
We assume that $(Y_k)$ is the observation process of Poisson threshold model as described in  \autoref{example:Poisson:threshold:defi} with $(\omega,a,b,c,d)=(\omega_\star,a_\star,b_\star,c_\star,d_\star)=\thv$.
\begin{proposition}
Assume that $\thv \in \Theta$ and that $(L,U) \cap \nset \neq \emptyset$. Then, for all $x \in\Xset$,
$\lim_{n \to \infty} \mlY{n,x}=\theta_\star$, $\PP^\thv\as$.
\end{proposition}
\begin{proof} Let $\{(X_k,Y_k)\, , \, k \in \zset\}$ satisfying the recursions given by \autoref{example:Poisson:threshold:defi} with $(\omega,a,b,c,d)=(\omega_\star,a_\star,b_\star,c_\star,d_\star)=\thv$.
 \autoref{prop:threshold:unique} shows that
\refhyps[C]{assum:stat:Y}{assum:moment:Y} hold so that \autoref{thm:PoissThres:misspecified-Consist-MLE} applies. It thus remains to show that $\Theta_\star=\{\thv\}$. This follows from \autoref{prop:ident} provided we show that
\begin{enumerate}[(a)]
\item \label{eq:X} $X_0=\f[\thv]{\chunk{Y}{-\infty}{0}}$
\item \label{eq:H} $H(x;\cdot)=H(x';\cdot)$ implies that $x=x'$,
\item \label{eq:f}$\f[\thv]{\chunk{Y}{-\infty}{0}}=\f[\theta]{\chunk{Y}{-\infty}{0}},\eqsp\PP^\thv\as$,  implies that $\theta=\thv$,
\end{enumerate}
Define
$$
a_\star(y)=a_\star+c_\star \1\{y \notin (L,U)\}\, , \quad b_\star(y)=b_\star y+d_\star y\1\{y \notin (L,U)\}\eqsp,
$$
First note that for all $m\geq 0$,
$$
X_0=\f[\thv]{\chunk{Y}{-m}{0}}(X_{-m})= X_{-m}\prod_ {\ell=-m}^0 a_\star(Y_\ell)+\sum_{j=0}^{m} [\omega_\star + b_\star(Y_{-j})] \prod_{\ell=1}^{j}a_\star(Y_{-\ell})  \eqsp.
$$
By applying \autoref{lem:useful} and  \autoref{prop:threshold:unique} , we have
$$
X_{-m}\prod_ {\ell=-m}^0 a_\star(Y_\ell) \leq \bar{\alpha}^{m+1}X_{-m} \to_{m \to \infty}0\,, \quad \PP^\thv\as
$$
so that
\begin{equation}
\label{eq:X0:Y}
X_0= \lim_{m \to \infty} \f[\thv]{\chunk{Y}{-m}{0}}(X_{-m}) = \f[\thv]{\chunk{Y}{-\infty}{0}}\, , \quad \PP^\thv\as
\end{equation}
where $\f[\thv]{\chunk{Y}{-\infty}{0}}$ is defined in \eqref{eq:def-stat-PoissThres}. Thus, \eqref{eq:X} holds. \eqref{eq:H} also clearly holds since $H(x,\cdot)$ is a Poisson distribution of parameter $x$. It remains to check \eqref{eq:f}. If $\PP^\thv\as$,  $\f[\thv]{\chunk{Y}{-\infty}{0}}=\f{\chunk{Y}{-\infty}{0}}$, then, by stationarity of the \sequence{Y}, we have: for all $t\in \zset$, $X_t=X'_t\, ,\quad \PP^\thv\as$
where we set $X'_t\eqdef\f{\chunk{Y}{-\infty}{t}}$. This implies that $X'_t=f^\theta_{Y_t} \circ f^\theta_{Y_{t-1}}(X'_{t-2})= f^\theta_{Y_t} \circ f^\theta_{Y_{t-1}}(X_{t-2})$, $\PP^\thv\as$, so that,
\begin{multline*}
\omega + a^\theta(Y_t) \left[ \omega +a^\theta(Y_{t-1})X_{t-2}+b^\theta(Y_{t-1})\right]+b^\theta(Y_{t})\\
=\omega_\star + a_\star(Y_t) \left[ \omega_\star+a_\star(Y_{t-1})X_{t-2}+b_\star(Y_{t-1})\right]+b_\star(Y_{t}) \eqsp, \quad \PP^\thv\as
\end{multline*}
Since $\CPPu[\thv]{(Y_{t-1},Y_{t})=(k,\ell)}{\chunk{Y}{-\infty}{t-2}} \neq 0$, for all $(k,\ell) \in \nset^2$ and $X_{t-2}$ is $\sigma^\theta(Y_{\ell}\, , \ell \leq t-2)$-measurable, we obtain that, for all $(k,\ell) \in \nset^2$, $\PP^\thv\as$,
\begin{multline*}
\omega + a^\theta(k) \left[ \omega +a^\theta(\ell)X_{t-2}+b^\theta(\ell)\right]+b^\theta(k)\\
=\omega_\star + a_\star(k) \left[ \omega_\star +a_\star(\ell)X_{t-2}+b_\star(\ell)\right]+b_\star(k)
\end{multline*}
Fix $\ell \in \nset$. Then, recalling that $a^\theta(k)$ is bounded in $k$ and that $b^\theta(k) \sim_{k \to \infty} b k$, we obtain that $b=b_\star$. Fix now $k \in \nset$ and take the equivalent of the previous equation as $\ell$ goes to infinity, we then obtain $a^\theta(k)b \ell=a_\star(k)b_\star \ell$ for all $k \in \nset$ which can also be written as
$$
a+c\1_{k \notin (L,U)}=a_\star+c_\star\1_{k \notin (L,U)}
$$
so that $a=a_\star$ and $c=c_\star$ by using that $(L,U) \cap \nset\neq \emptyset$. Finally, using $b=b_\star$, $a=a_\star$ and $c=c_\star$, we have $\PP^\thv\as$ for all $k \in \nset$,
$$
\omega + a_\star(k) X_{t-1}+b_\star k + d k\1_{k \notin (L,U)}=\omega_\star + a_\star(k) X_{t-1}+b_\star k + d_\star k \1_{k \notin (L,U)}
$$
so that
$$
\omega +dk \1_{k \notin (L,U)}=\omega_\star +d_\star k\1_{k \notin (L,U)}\eqsp.
 $$
 which again implies that $\omega=\omega_\star$ and $d=d_\star$ since $(L,U) \cap \nset \neq \emptyset$.
\end{proof}

\subsection{The Log-linear Poisson autoregression in misspecified Models}
Let $\Theta$ be the following (compact) set of parameters
\begin{equation}
\label{eq:def-parameter-set}
\Theta= \set{\theta= (d,a,b) \in \rset^3}{|d|\leq \tilde{d},\quad |a| \leq \tilde{a}<1, \quad |b|\leq \tilde{b}} \eqsp.
\end{equation}
where $\tilde{d},\ \tilde{a},\ \tilde{b}$ are positive constants of $\rset$.
 Assume that the observations $(Y_n)_{n \in \zset}$ are integer-valued and satisfy the assumptions \refhyps[C]{assum:stat:Y}{assum:moment:Y}.
The Log-linear Poisson autoregression model described in \eqref{defi:logPoisson} may be rewritten as in \eqref{eq:def:XY:theta}, by setting $\Xset=\rset$, $\Yset=\nset$, $\theta=(d,a,b)$, and
\begin{align}
&f^\theta_y(x)=d+ax+b \ln(1+y) \eqsp,\label{eq:logPoiss:def:f}\\
&h(x;y)=\frac{\rmd H(x;\cdot)}{\rmd \mu}(y)=\exp(-\rme^x) \rme^{xy}/{y!} \eqsp,\label{eq:logPoiss:def:h}
\end{align}
where $\mu$ is the counting measure on $\nset$. Using \eqref{eq:notationItere:f}, we have for all $s\leq t$,
\begin{equation}
\label{eq:logPoiss:itere:f}
\f{\chunk{y}{s}{t}}(x)=d \frac{1-a^{t-s+1}}{1-a}+a^{t-s+1}x+b\sum_{j=0}^{t-s} a^j \ln(1+y_{t-j})\eqsp.
\end{equation}
With these definitions, let $\theta_{n,x}$ be the Maximum Likelihood estimator associated to the likelihood function $\lkdM[n,x]{\chunk{Y}{1}{n}}[\theta]$ as defined in  \eqref{eq:defi:mle} and \eqref{eq:defi:lkdM}.

\begin{theorem} \label{thm:misspecified-Consist-MLE}
Assume \refhyps[C]{assum:stat:Y}{assum:moment:Y}. Then, for all $x \in \Xset$,
$\lim_{n \to \infty} \mathsf{d}(\mlY{n,x}, \Theta_\star) = 0$, $\PP_\thv\as$,
where
\begin{align}
&\Theta_\star\eqdef\argmax_{\theta \in \Theta} \PE_\star\left( Y_0 \f{\chunk{Y}{-\infty}{-1}}-\rme^{\f{\chunk{Y}{-\infty}{-1}}}-\ln Y_0!\right)\eqsp, \label{eq:def-Theta-star-Y}\\
&\f{\chunk{Y}{-\infty}{n}} \eqdef \frac{d}{1-a} + b \sum_{j=0}^\infty a^j \ln( 1 + Y_{n-j} )\eqsp,\quad \forall (n,\theta) \in \zset \times \Theta\eqsp. \label{eq:def-stat-log-intensity}
\end{align} 
\end{theorem}
\begin{proof}
 According to \autoref{thm:consist:misspecified}, it is sufficient to check \refhyps[B]{assum:continuity:Y}{assum:limit:Y}. \refhyp[B]{assum:continuity:Y} clearly holds.
 Using \autoref{rem:integerValued}, since $\Yset=\nset$, we only need to check \refhyp[B]{assum:limit:Y}-\eqref{item:lim:moins:infty} and  \refhyp[B]{assum:limit:Y}-\eqref{item:lim:k}.
First note that
\begin{equation}\label{eq:pratique}
\sup_{\theta \in \Theta} |\f{\chunk{Y}{-\infty}{0}}|\leq \tilde{d}/(1-\tilde{a})+\tilde{b} \sum_{j=0}^\infty {\tilde{a}}^j \ln( 1 + Y_{-j}  )<\infty,\quad \PP_\star\as
\end{equation}
which is finite according to \refhyp[C]{assum:moment:Y} by using \autoref{lem:useful}. Now, write for all $\theta=(d,a,b) \in \Theta$,
\begin{align*}
&|\f{\chunk{Y}{-m}{0}}(x)-\f{\chunk{Y}{-\infty}{0}}|\\
&\quad =|a|^{m+1}\left|-\frac{d}{1-a}+x+b \sum_{\ell=0}^\infty a^\ell \ln(1+Y_{-m-1-\ell}) \right|\\
&\quad \leq |\tilde a|^{m+1}\left(\frac{\tilde d}{1-\tilde a}+|x|+\tilde b \sum_{\ell=0}^\infty \tilde{a}^\ell \ln(1+Y_{-m-1-\ell}) \right)\eqsp.
\end{align*}
By \refhyp[C]{assum:moment:Y} and by applying \autoref{lem:useful}, the \rhs\ (which does not depend on $\theta$) converges to 0 as $m$ goes to infinity. Thus, \eqref{item:lim:moins:infty} holds. We now turn to \eqref{item:lim:k}.
\begin{multline}
\sup_{\theta \in \Theta} |\ln h(\f{\chunk{Y}{1}{k-1}}(x);Y_k)- \lkdStat{\chunk{Y}{-\infty}{k}}[\theta]|\leq \\
 Y_k \sup_{\theta \in \Theta}\left| \f{\chunk{Y}{1}{k-1}}(x) - \f{\chunk{Y}{-\infty}{k-1}} \right| +\sup_{\theta \in \Theta}\left| \rme^{\f{\chunk{Y}{1}{k-1}}(x)} - \rme^{\f{\chunk{Y}{-\infty}{k-1}}} \right|\eqsp.
 \label{eq:bound-unif-log-lkh-Y}
\end{multline}
Consider the first term in the rhs. It follows immediately from \eqref{eq:logPoiss:itere:f} and \eqref{eq:def-stat-log-intensity} that
\begin{equation}
\label{eq:difference-log-intensity}
 \f{\chunk{Y}{1}{k-1}}(x) - \f{\chunk{Y}{-\infty}{k-1}}= a^{k-1} \left( x - \f{\chunk{Y}{-\infty}{0}} \right) \eqsp.
\end{equation}
This implies that, for all $k \geq 1$,
\begin{equation*}
Y_k \sup_{\theta \in \Theta}\left| \f{\chunk{Y}{1}{k-1}}(x) - \f{\chunk{Y}{-\infty}{k-1}} \right| \leq Y_k \ \tilde{a}^{k-1}(x +\sup_{\theta \in \Theta} |\f{\chunk{Y}{-\infty}{0}}|) \eqsp,
\end{equation*}
which converges $\PP\as$ to 0 as $k$ goes to infinity according to \eqref{eq:pratique} and by applying \autoref{lem:useful} under  \refhyp[C]{assum:moment:Y}. Moreover, \eqref{eq:difference-log-intensity} also implies that
\begin{multline*}
\left| \rme^{\f{\chunk{Y}{0}{k-1}}(x)} - \rme^{\f{\chunk{Y}{-\infty}{k-1}}} \right|
=  \rme^{\f{\chunk{Y}{-\infty}{k-1}}} \left| \rme^{a^{k-1} \left( x - \f{\chunk{Y}{-\infty}{0}} \right)} - 1 \right| \\
\leq |a|^{k-1}   \left| x - \f{\chunk{Y}{-\infty}{0}} \right| \rme^{\left| x - \f{\chunk{Y}{-\infty}{0}} \right|+\f{\chunk{Y}{-\infty}{k-1}}} \eqsp,
\end{multline*}
so that the second term of the rhs of \eqref{eq:bound-unif-log-lkh-Y} is bounded according to
\begin{multline*}
\sup_{\theta \in \Theta}\left| \rme^{\f{\chunk{Y}{0}{k-1}}(x)} - \rme^{\f{\chunk{Y}{-\infty}{k-1}}} \right| \\
 \leq \sup_{\theta \in \Theta}\left( \left| x - \f{\chunk{Y}{-\infty}{0}} \right| \rme^{\left| x - \f{\chunk{Y}{-\infty}{0}} \right|} \right)  \times  \tilde{a}^{k-1} \sup_{\theta \in \Theta} \rme^{|\f{\chunk{Y}{-\infty}{k-1}}|}\eqsp.
\end{multline*}
To complete the proof, it is thus sufficient to show that
$$
\lim_{k \to \infty} \tilde{a}^k \exp\left\{\sup_{\theta \in \Theta} |\f{\chunk{Y}{-\infty}{k-1}}|\right\} = 0,  \quad \PP_\thv\as
$$
But this is straightforward by applying \autoref{lem:useful} since by  \eqref{eq:pratique} and by setting $V_k\eqdef \exp\left\{\sup_{\theta \in \Theta} |\f{\chunk{Y}{-\infty}{k-1}}|\right\}$, we have
\begin{multline*}
\PE_\star\left[  \left(\ln V_1 \right)_+\right]
 \leq \frac{\tilde{d}}{1-\tilde{a}}+\tilde{b} \sum_{j=0}^\infty {\tilde{a}}^j \PE_\star\left[ \ln( 1 + Y_{-j}  ) \right]= \frac{\tilde{d}+\tilde{b} \PE_\star[\ln(1+Y_{0})]}{1-\tilde{a}}\eqsp,
\end{multline*}
which is finite by \refhyp[C]{assum:moment:Y}.
\end{proof}

\subsection{Log-linear Poisson autoregression in well-specified models}
Let $\Theta$ be the following (compact) set of parameters
\begin{equation}
\label{eq:def-parameter-set}
\Theta= \{ (d,a,b) \in \rset^3;\ |d|\leq \tilde{d},\quad |a+b|\vee |a|\vee |b|\leq \tilde{\alpha}<1\} \eqsp.
\end{equation}
where $(\tilde{d},\ \tilde{\alpha})\in \rset^+_* \times (0,1)$.
We assume that \sequence{Y} is the observation process of the Log-linear Poisson autoregression model described in \eqref{defi:logPoisson} with parameters $(d,a,b)=(d_\star,a_\star,b_\star)=\thv$.
\begin{proposition}
Assume that $\thv \in \Theta$. Then, for all $x \in\Xset$,
$$
\lim_{n \to \infty} \mlY{n,x}=\theta_\star\, , \quad \PP^\thv\as
$$
\end{proposition}
\begin{proof}
Let $\{(X_k,Y_k)\, , \, k \in \zset\}$ satisfying \eqref{defi:logPoisson} with $(d,a,b)=(d_\star,a_\star,b_\star)=\thv$.
 \autoref{prop:logPoiss:unicMesInv} shows that
\refhyps[C]{assum:stat:Y}{assum:moment:Y} hold so that \autoref{thm:misspecified-Consist-MLE} applies. It thus remains to show that $\Theta_\star=\{\thv\}$. This follows from \autoref{prop:ident} provided we show that
\begin{enumerate}[(a)]
\item \label{eq:X} $X_0=\f[\thv]{\chunk{Y}{-\infty}{0}}$, $\PP^\thv\as$,
\item \label{eq:H} $H(x;\cdot)=H(x';\cdot)$ implies that $x=x'$,
\item \label{eq:f}$\f[\thv]{\chunk{Y}{-\infty}{0}}=\f[\theta]{\chunk{Y}{-\infty}{0}},\eqsp\PP^\thv\as$,  implies that $\theta=\thv$,
\end{enumerate}
First note that for all $m\geq 0$,
$$
X_0=\f[\thv]{\chunk{Y}{-m}{0}}(X_{-m})=d_\star \frac{1-a_\star^{m+1}}{1-a_\star}+a_\star^{m+1}X_{-m}+b_\star\sum_{j=0}^{m}a_\star^{j}\ln(1+Y_{-j-1}) \eqsp.
$$
By applying \autoref{lem:useful} and \autoref{prop:logPoiss:unicMesInv}, we have $\lim_{m \to \infty}  a_\star^{m+1}X_{-m}=0\, , \, \PP^\thv\as$, so that
\begin{equation}
\label{eq:X0:Y}
X_0= \lim_{m \to \infty} \f[\thv]{\chunk{Y}{-m}{0}}(X_{-m}) = \f[\thv]{\chunk{Y}{-\infty}{0}}\, , \quad \PP^\thv\as
\end{equation}
where $\f[\thv]{\chunk{Y}{-\infty}{0}}$ is defined in \eqref{eq:def-stat-log-intensity}. Thus, \eqref{eq:X} holds. \eqref{eq:H} also clearly holds since $H(x,\cdot)$ is a Poisson distribution of parameter $\rme^x$. It remains to check \eqref{eq:f}. If $\PP^\thv\as$,  $\f[\thv]{\chunk{Y}{-\infty}{0}}=\f{\chunk{Y}{-\infty}{0}}$, then, by definition of $\f{\chunk{Y}{-\infty}{0}}$,
$$
\frac{d_\star}{1-a_\star} -\frac{d}{1-a}+  \sum_{j=0}^\infty (b_\star a_\star^j -b  a^j)\ln( 1 + Y_{-j} )=0\, , \quad \PP^\thv\as
$$
Conditionally on $\sigma(Y_m; m \leq -1)$, $Y_{0}$ is a Poisson random variable with a positive intensity; thus, the lhs is constant only if $b_\star=b$. This implies that
$$
\frac{d_\star}{1-a_\star} -\frac{d}{1-a}+ b \sum_{j=1}^\infty ( a_\star^j - a^j)\ln( 1 + Y_{-j} )=0\, , \quad \PP^\thv\as
$$
By the same argument, the lhs is constant conditionally on $\sigma(Y_m; m \leq -2)$ only if $a_\star=a$. In that case, the previous equality writes: $d_\star-d=0$ which completes the proof.
\end{proof}

\section{Proofs of \autoref{thm:uniqueInvariantProbaMeasure}, \autoref{lem:piV} and
\autoref{lem:practical:conditions}}
\label{sec:proof}

The proof roughly follows the lines of \cite{henderson:matteson:woodard:2011} with the difference that we relax the Lipshitz assumption and introduce a drift function $W$.
 In all this section, $(\Xset,d)$ is a Polish (complete, separable and metric) space and denote by $\Xsigma$ its associated Borel $\sigma$-field. A {\em totally separating system of metrics} $\{d_n\, ,\ n \in \nset\}$
for $\Xset$ is a set of metrics such that for all fixed $x,x' \in \Xset$, the sequence $\{d_n(x,x')\, , \  n \in \nset\}$
is nondecreasing in $n$ and $\lim_{n\to \infty} d_n (x,x') =\1_{x\neq x'}$. A metric $d$ on $\Xset$ induces a Wasserstein distance between probability measures $µ_1$ and $µ_2$ on $(\Xset,\Xsigma)$ defined by:
\begin{equation}\label{eq:defi:Wasserstein}
\wass[d]{µ_1 - µ_2} =\inf \left\{ \int \mu(\rmd x,\rmd x') d(x,x'):\, \mu \in \mcm(\mu_1,\mu_2)\right\}\eqsp,
\end{equation}
where $\mcm(\mu_1,\mu_2)$ is the set of probability measures $\mu$ on $(\Xset^2,\Xsigma^{\otimes 2})$ such that
$$
\mu(A\times \Xset)=\mu_1(A)\, , \quad \mu( \Xset \times A)=\mu_2(A)\eqsp, \quad
\text{for all $A \in \Xsigma$.}
$$
Since $\Xset$ is a separable metric space, the Kantorowich-Rubinstein duality theorem applies (see for example \cite[Theorem 11.8.2]{dudley:2002}) and we have
\begin{equation} \label{eq:dualityKantoRubin}
\wass[d]{µ_1 - µ_2}= \sup \left\{ µ_1(f)- µ_2(f)\, :\ \normLip[d]{f}\leq 1 \right\}\eqsp,
\end{equation}
where
$$
\normLip[d]{f}=\sup\left\{ \frac{|f(x) - f(x')|}{d(x,x')}\, : \ x,x'\in \Xset\, , \ x\neq x'\right\}\eqsp.
$$
Recall the definition of an asymptotically strong Feller kernel, first introduced by \cite{hairer:mattingly:2006}:
\begin{definition}
A Markov kernel $Q$ is asymptotically strong Feller if, for all  $x\in \Xset$, there exist a totally
separating system of metrics $\{d_n, n \in \nset\}$ for $\Xset$ and a sequence of integers $\{t_n, n \in \nset\}$ such that
$$
\lim_{\gamma \to 0} \limsup_{n \to \infty} \sup_{x' \in \ball{x}{\gamma}} \wass[d_n]{Q^{t_n}(x,\cdot)-Q^{t_n}(x',\cdot)}=0\eqsp.
$$
where $ \ball{x}{\gamma}$ is the open ball of radius $\gamma$ \wrt\ $d$ and centered at $x$.
\end{definition}
The following theorem is taken from \cite{hairer:mattingly:2006} and provide conditions for obtaining uniqueness of the invariant probability measure.
\begin{theorem} \label{thm:hairerMattingly}
Assume that the Markov kernel $Q$ is asymptotically strong Feller and admits a reachable
point $x\in \Xset$. Then, $Q$ has at most one stationary distribution.
\end{theorem}

\begin{proof}[Proof of \autoref{thm:uniqueInvariantProbaMeasure}]
Under \refhyp[A]{assum:weakFeller}, \cite[Theorem 2]{tweedie:1988} show that $Q$ admits at least one stationary distribution. Since by \refhyp[A]{assum:reachable} $Q$ admits a reachable point, we conclude by applying \autoref{thm:hairerMattingly} provided that we can prove $Q$ is asymptotically strong Feller. Denote \begin{equation}
\label{eq:definition-T}
T =\inf\set{i \in \nset}{U_i=0}
\end{equation}
with the convention $\inf \emptyset =\infty$. We preface the proof by the following technical lemma:
\begin{lemma} \label{lem:passageEversEetoile}
Let $\{(X_k,X'_k,U_k)\, , k \in\nset\}$ be a Markov chain on $(\Xset^2\times\{0,1\})$ with Markov kernel $\bar Q$, introduced in \refhyp{assum:asympStrgFeller:general}.  Then, for all real-valued nonnegative measurable function $\varphi$ on $\Xset^2$, $n \in \nset^*$, $x,x' \in \Xset$ and $u \in [0,1]$,
\begin{equation}\label{eq:passageEversEetoile}
 \Ecan[\delta_x \otimes \delta_{x'} \otimes \mathcal{B}(u)][\bar Q]\left[\varphi(X^\sharp_n) \1_{\{T>n\}}\right]=  \Ecan[\delta_x \otimes \delta_{x'} ][Q^\sharp]\left[\varphi(X^\sharp_n) \prod_{i=0}^{n-1} \alpha(X^\sharp_i)\right]\eqsp,
\end{equation}
where $Q^\sharp$ and $\alpha$ are introduced in \refhyp{assum:asympStrgFeller:general}, $X^\sharp_i\eqdef (X_i,X'_i)$ and $\mathcal{B}(u)$ is the Bernoulli distribution with parameter $u$.

\end{lemma}

\begin{proof}
The proof is by induction. Note first that \eqref{eq:passageEversEetoile} obviously holds for $n=1$. Now, assume that \eqref{eq:passageEversEetoile} holds for some $n\geq1$. Then, noting that $\1_{\{T>n+1\}}=\prod_{i=1}^{n+1} U_i$ (where $T$ is defined in \eqref{eq:definition-T})
\begin{align*}
 &\Pcan[\delta_x\otimes\delta_{x'}\otimes{\mathcal B}(u) ][\bar Q]\left[\varphi(X^\sharp_{n+1}) \1_{\{T>n+1\}}\right]=\Ecan[\delta_x\otimes\delta_{x'}\otimes{\mathcal B}(u) ][\bar Q]\left[\varphi(X^\sharp_{n+1}) \prod_{i=1}^{n+1} U_i\right]\\
 &\quad =\Ecan[\delta_x\otimes\delta_{x'}\otimes{\mathcal B}(u) ][\bar Q]\left[ \Ecan[][\bar Q](U_{n+1}\varphi(X^\sharp_{n+1}) \vert X^\sharp_n,U_n) \prod_{i=1}^{n} U_i\right]\\
 &\quad =\Ecan[\delta_x\otimes\delta_{x'}\otimes{\mathcal B}(u) ][\bar Q]\left[  \alpha(X^\sharp_n)\ Q^\sharp \varphi(X^\sharp_{n}) \prod_{i=1}^{n} U_i\right] \eqsp.
 \end{align*}
 where the last equality follows from \eqref{eq:barQ:alpha:Qstar}.
Applying the induction assumption to the \rhs\ of the inequality, we obtain
 \begin{align*}
& \Pcan[\delta_x\otimes\delta_{x'}\otimes{\mathcal B}(u) ][\bar Q]\left[\varphi(X^\sharp_{n+1}) \1_{\{T>n+1\}}\right]=\Ecan[\delta_x\otimes \delta_{x'}][Q^\sharp]\left[ \alpha(X^\sharp_n) \ Q^\sharp \varphi(X^\sharp_{n}) \prod_{i=0}^{n-1} \alpha(X^\sharp_i)\right]\\
&\quad =\Ecan[\delta_x\otimes \delta_{x'}][Q^\sharp]\left[ \alpha(X^\sharp_n) \, \Ecan[][Q^\sharp](\varphi(X^\sharp_{n+1})\vert X^\sharp_n) \prod_{i=0}^{n-1} \alpha(X^\sharp_i)\right] \\
& \quad =\Ecan[\delta_x\otimes \delta_{x'}][Q^\sharp]\left[ \varphi(X^\sharp_{n+1})\prod_{i=0}^{n} \alpha(X^\sharp_i)\right]\eqsp.
\end{align*}
The proof is completed.
\end{proof}
Now, consider $d_n(x,x')= 1 \wedge [n d(x,x')]$. Obviously, for all fixed $x,x' \in \Xset$, the sequence $\{d_n(x,x')\, , \ n \in \nset\}$ is  nondecreasing and $\lim_{n \to \infty} d_n(x,x')=\1\{x \neq x'\}$ so that $\{d_n\, , \ n \in \nset\}$ is a totally separating system of metrics. Moreover, the Kantorovich-Rubinstein duality theorem \eqref{eq:defi:Wasserstein}, \eqref{eq:dualityKantoRubin} and the marginal conditions \eqref{eq:marginalConditionsQ} yield:  for all $(x,x',u) \in \Xset\times  \ball{x}{\gamma_x}\times [0,1]$,
\begin{multline}
\label{eq:asf:technicos:one}
\wass[d_n]{\delta_x Q^n -\delta_{x'}Q^n} \leq \Ecan[\delta_x\otimes\delta_{x'}\otimes{\mathcal B}(u) ][\bar Q](d_n(X_n,X'_n)) \\
\leq \Pcan[\delta_x\otimes\delta_{x'}\otimes{\mathcal B}(u) ][\bar Q](T\leq n)+  \Ecan[\delta_x\otimes\delta_{x'}\otimes{\mathcal B}(u) ][\bar Q](d_n(X_n,X'_n) \1\{T>n\})\eqsp,
\end{multline}
where we have used that $d_n(x,y) \leq 1$. First consider the second term of the \rhs. Applying \autoref{lem:passageEversEetoile}, combined with $d_n(x,y) \leq n d(x,y)$, $\alpha(x,x')\leq 1$ for all $(x,x') \in \Xset^2$ and \eqref{eq:majo:d}, yields
\begin{align}
\label{eq:asf:technicos:two}
\Ecan[\delta_x\otimes\delta_{x'}\otimes{\mathcal B}(u) ][\bar Q](d_n(X_n,X'_n) \1\{T>n\})&\leq  n \Ecan[\delta_{x}\otimes \delta_{x'}][ Q^\sharp]\left[d(X_n,X'_n) \prod_{i=0}^{n-1} \alpha(X_i,X'_i)\right] \nonumber\\\leq  n D \rho^n d(x,x')\eqsp.
\end{align}
We now turn to the first term of the \rhs\ of \eqref{eq:asf:technicos:one}. By \eqref{eq:hyp:beta:gene}, we get
\begin{multline*}
\Pcan[\delta_x\otimes\delta_{x'}\otimes{\mathcal B}(u) ][\bar Q](T\leq n)=\sum_{k=0}^{n-1} \Pcan[\delta_x\otimes\delta_{x'}\otimes{\mathcal B}(u) ][\bar Q](T> k,\ U_{k+1}=0)\\
\leq \sum_{k=0}^{n-1} \Ecan[\delta_x\otimes\delta_{x'}\otimes{\mathcal B}(u) ][\bar Q][\1\{T> k\}d(X_{k},X'_{k})W(X_{k},X'_{k})]\\
\end{multline*}
Applying \autoref{lem:passageEversEetoile} to the \rhs\ , combined with $\alpha\leq 1$ and \eqref{eq:majo:d:w},
we obtain
\begin{multline*}
\Pcan[\delta_x\otimes\delta_{x'}\otimes{\mathcal B}(u) ][\bar Q](T\leq n)\leq
\sum_{k=0}^{n-1} \Ecan[\delta_x\otimes\delta_{x'}][Q^\sharp]\left[d(X_{k},X'_{k})W(X_{k},X'_{k})\prod_{i=0}^{k-1} \alpha(X_{i},X'_{i})\right] \\
\leq D d^{\zeta_1}(x,x') W^{\zeta_2}(x,x')\sum_{k=0}^{n-1} \rho^k \leq  \frac{D d^{\zeta_1}(x,x') W^{\zeta_2}(x,x')}{1-\rho} \eqsp.
\end{multline*}
Plugging this and \eqref{eq:asf:technicos:two} into \eqref{eq:asf:technicos:one} yields: for all $x \in \Xset$ and all $x' \in \ball{x}{\gamma}$ where $\gamma <\gamma_x$,
$$
\wass[d_n]{\delta_x Q^n -\delta_{x'}Q^n} \leq D \left( n \rho^n \gamma+  \gamma^{\zeta_1} \sup_{y \in \ball{x}{\gamma_x} }W^{\zeta_2}(x,y)/(1-\rho) \right) \eqsp,
$$
where $\gamma_x$ is defined in \eqref{eq:definition-gamma-x}. Thus, for all $x \in \Xset$,
$$
\lim_{\gamma \to 0} \limsup_{n \to \infty} \sup_{x' \in \ball{x}{\gamma}} \wass[d_n]{Q^{n}(x,\cdot)-Q^{n}(x',\cdot)}=0\eqsp.
$$
The proof is completed.
\end{proof}

\begin{proof}[Proof of \autoref{lem:piV}]
Since for all $M>0$, the function $x \mapsto x \wedge M$ is concave, we have for all $n \in \nset$,
$$
Q^n (V \wedge M) \leq (Q^n V) \wedge M \leq [\lambda^n V +b/(1-\lambda)] \wedge M \eqsp.
$$
By integrating with respect to $\pi$, we obtain that
$$
\pi (V\wedge M) = \pi Q^n (V \wedge M) \leq \pi \left\{[\lambda^n V +\beta/(1-\lambda)] \wedge M\right\}\eqsp.
$$
The Lebesgue convergence theorem yields by letting $n$ goes to infinity
$$
\pi (V\wedge M) \leq \beta/(1-\lambda) \wedge M\eqsp.
$$
The proof follows by letting $M$ goes to infinity.
\end{proof}

\appendix

\section{Ergodicity of one-sided and two-sided sequences}
Let $(\Xset,\Xsigma)$ be a measurable space.
Denote by $\shift:\Xset^\nset \rightarrow \Xset^\nset $ and $\tilde \shift:\Xset^\zset \rightarrow \Xset^\zset $ the shift operators defined by: for all  $\mathbf{x}=(x_t)_{t \in \nset} \in \Xset^\nset$ and all $\mathbf{\tilde x}=(\tilde x_t)_{t \in \zset} \in \Xset^\zset$,
\begin{align}\label{eq:def-shift}
\shift(\mathbf{x})&= (y_t)_{t \in\nset},\quad \mbox{where} \quad y_t=x_{t+1},\quad \forall t \in \nset \eqsp,\\
\tilde \shift(\mathbf{\tilde x})&= (\tilde y_t)_{t \in\zset},\quad \mbox{where} \quad \tilde y_t=\tilde x_{t+1},\quad \forall t \in \zset \eqsp.
\end{align}
Note that $\tilde \shift$ is invertible while $\shift$ is not.
Let $P$ be a Markov kernel on $(\Xset, \Xsigma)$. Denote by $\PP_\mu$ the probability induced on $(\Xset^\nset,\Xsigma^{\otimes \nset})$ by a Markov chain of  initial distribution $\mu$ and Markov kernel $P$ and write $\PE_\mu$ the associated expectation operator. If $\mu=\pi$ is an invariant distribution for $P$, we can define a probability $\tilde \PP_\pi$ induced on $(\Xset^\zset,\Xsigma^{\otimes \zset})$ by the Markov kernel $P$ and initial distribution $\pi$. Similarly, we write $\tilde \PE_\pi$ the associated expectation operator. Moreover, $\tilde \PP_\pi$ extends $\PP_\pi$ on $\zset$ in the sense that for all $A\in \Xsigma^{\otimes \nset}$, $\PP_\pi(A)=\tilde \PP_\pi(\Xset^{\zset_-^*} \times A)$, which can also be written as
\begin{equation}
\label{eq:relationTildePetP}
\PP_\pi=\tilde \PP_\pi \circ p^{-1} \eqsp,
\end{equation}
 where $p$ is the mapping from $(\Xset^\zset,\Xsigma^{\otimes \zset})$ to $(\Xset^\nset,\Xsigma^{\otimes \nset})$ defined by
\begin{equation}\label{eq:defiF}
p(\omega)=(\omega_n)_{n \in \nset} \quad \mbox{where}\quad \omega=(\omega_n)_{n \in \zset}\eqsp.
\end{equation}
Define for all $k \in \nset$, $X_k: \Xset^\nset \to \Xset$ by
$$
X_k(\omega)=\omega_k \, , \quad \mbox{where} \quad \omega=(\omega_\ell)_{\ell \in \nset}\in \Xset^{\nset}\eqsp,
 $$
and similarly, define for all $k \in \zset$, $\tilde X_k: \Xset^\zset \to \Xset$ by
\begin{equation}
\label{eq:defiTildeXk}
\tilde X_k(\tilde \omega)=\tilde \omega_k\, , \quad \mbox{where}\quad \tilde \omega=(\omega_\ell)_{\ell \in \zset}\in \Xset^{\zset}\eqsp.
\end{equation}
Recall that $(\Omega,\mcf, \PP,\tau)$ is a {\em measure-preserving} dynamical system if $(\Omega,\mcf, \PP)$ is a probability space and $\tau: \Omega \to \Omega$ is measurable such that $\PP \circ \tau^{-1}=\PP$. Moreover, a measure-preserving dynamical system $(\Omega,\mcf, \PP,\tau)$ is said to be {\em ergodic} if for all invariant subset $A \in \mcf$, i.e. $\1_A=\1_A \circ \shift$, we have $\PP(A)=0$ or $1$. Recall that if $\1_B=\1_B \circ \shift\, , \eqsp \PP\as$, then, there exists an invariant set $A$ such that $\1_A=\1_B\,,\eqsp \PP\as$ In the following, $\tau^k: \Omega \to \Omega$ is the mapping $\tau$ iterated $k$ times, that is $\tau^k=\tau \circ \ldots \circ \tau$ and by convention $\tau^0(\omega)=\omega$ for all $\omega \in \Omega$.
\begin{theorem} \label{thm:unicInvMeasImplyErgodic}
Assume that the Markov kernel $P$ has a unique stationary distribution $\pi$. Then, the dynamical system $(\Xset^\nset,\Xsigma^{\otimes \nset}, \PP_\pi,\shift)$ is ergodic.
\end{theorem}
\begin{proof}
Let $A \in \Xsigma^{\otimes \nset}$ be an invariant set for $(\Xset^\nset,\Xsigma^{\otimes \nset}, \PP_\pi,\shift)$, that is:
$\1_A =\1_A \circ \shift$.
We will show that $\PP_\pi(A)=0$ or $1$ by contradiction. Assume indeed that $\PP_\pi(A) \in (0,\,1)$. Using the Markov property and the fact that $A$ is invariant,
$$
\PE_{X_k}(\1_A)=\cesp[\pi]{\1_A \circ \shift^k}{\mcf_k} = \cesp[\pi]{\1_A }{\mcf_k} \,, \quad \PP_\pi\as\
$$
where $\mcf_k=\sigma(X_0,\ldots,X_k)$. Therefore, $\{ (\PE_{X_k}(\1_A),\mcf_k), k \in \nset \}$ is a uniformly integrable martingale.
By  \cite[Corollary 2.2]{hall:heyde:1980}, $\lim_{k \to \infty} \PE_{X_k}(\1_A) = \1_A$, $\PP_\pi$\as\ and $\lim_{k  \to \infty} \PE_{\pi}|\PE_{X_k}(\1_A) -\1_A|=0$. Then,
\begin{multline*}
\PE_\pi(|\1_A-\PE_{X_0}(\1_A)|)=\PE_\pi(|\1_A-\PE_{X_0}(\1_A)|\circ \shift^k)\\
=\PE_\pi(|\1_A-\PE_{X_k}(\1_A)|)=\lim_{k \to \infty}\PE_\pi(|\1_A-\PE_{X_k}(\1_A)|)=0 \eqsp.
\end{multline*}
so that $\1_A = \PP_{X_0}(A)$, $\PP_\pi$\as\ Setting
\begin{equation}
\label{eq:defiGammaA}
\Gamma_A\eqdef\{x \in \Xset, \; \PP_{x}(A)=1\}\eqsp,
\end{equation}
we then obtain $\1_A=\1_{\Gamma_A}(X_0)$, $\PP_\pi$\as\ Combining it with the fact that $A$ is invariant, we get for all $k \in \nset$,
\begin{equation}
\label{eq:1_A=1_Gamma}
\1_A =\1_A \circ \shift^k=\1_{\Gamma_A}(X_0\circ \shift^k)= \1_{\Gamma_A}(X_k)\,, \quad \PP_\pi\as
\end{equation}

Now,  let $\pi_{A}(\cdot)= \alpha^{-1}\pi(\Gamma_A\cap \cdot)$ where $\alpha=\PP_\pi(A)\neq 0$. By definition of $\pi_A$ and by using \eqref{eq:1_A=1_Gamma} with $k=0$ and $k=1$, we get for all $B\in \Xsigma$,
\begin{align*}
\PP_{\pi_A}(X_1\in B)&=\alpha^{-1}\PP_\pi(\{X_1\in B\}\cap \{X_0\in \Gamma_A\})\\
&=\alpha^{-1}\PP_\pi(\{X_1\in B\}\cap \{X_1\in \Gamma_A\}) \\
&=\alpha^{-1}\PP_\pi(X_1\in B\cap \Gamma_A)=\alpha^{-1}\pi(B\cap \Gamma_A)
=\pi_A(B) \eqsp,
\end{align*}
showing that $\pi_A$ is a stationary distribution for the Markov kernel $Q$. Since $A$ is an invariant set, $A^c$ is also an invariant set and thus, $\pi_{A^c}$ is also a stationary distribution for the Markov kernel $Q$. Since by assumption there exists a unique stationary distribution, we have that $\pi_A=\pi_{A^c}$ which is not possible since these probability measures have disjoint supports (indeed by \eqref{eq:defiGammaA}, we have $\Gamma_A\cap \Gamma_{A^c}=\emptyset$).
\end{proof}

\begin{theorem} \label{thm:ergod:N:Z}
Assume that the dynamical system $(\Xset^\nset,\Xsigma^{\otimes \nset}, \PP_\pi,\shift)$ is ergodic. Then, the dynamical system $(\Xset^\zset,\Xsigma^{\otimes \zset}, \tilde \PP_\pi,\tilde \shift)$ is ergodic.
\end{theorem}

\begin{proof}

Let $A$ be an invariant set for the dynamical system $(\Xset^\zset,\Xsigma^{\otimes \zset}, \tilde \PP_\pi,\tilde \shift)$, that is $\1_A=\1_A \circ \tilde \shift$. We now show that $\tilde \PP_\pi(A) =0$ or $1$.

 Note first that $\Xsigma^{\otimes \zset}=\sigma(\mcf_{-k}\, , \, k \in \nset)$ where $\mcf_\ell=\sigma(\tilde X_i\,,\, \ell\leq i <\infty)$ and $\tilde X_i$ is defined in \eqref{eq:defiTildeXk}. This allows to apply the approximation Lemma (see for example \cite[Corollary 1.5.3]{gray:2009}) showing that for all $\epsilon>0$, there exists $k_\epsilon \in\nset$ and a $\mcf_{-k_\epsilon}$-measurable random variable $Z_\epsilon$ such that $\tilde \PE_\pi(|Z_\epsilon|)<\infty$ and
$\tilde \PE_\pi|\1_A-Z_\epsilon|\leq \epsilon$. Then, setting $Y_\epsilon=Z_\epsilon \circ \tilde \shift^{k_\epsilon} \in \mcf_0$ and using that $A$ is an invariant set, we obtain
$$
\tilde \PE_\pi|\1_A-Y_\epsilon|=\tilde \PE_\pi|\1_A \circ \tilde \shift^{k}-Z_\epsilon\circ \tilde \shift^{k}|=\tilde \PE_\pi|\1_A-Z_\epsilon|\leq \epsilon \eqsp.
$$
 The positive real number $\epsilon$ being arbitrary, there exists $Y$ such that $\tilde \PE_\pi|Y|<\infty$ and $\1_A=Y$, $\tilde \PP_\pi$\as\ which implies that $1=\tilde \PP_\pi (\1_A=Y)\leq \tilde \PP_\pi(Y \in \{0,1\})\leq 1$. Thus, there exists $B \in \mcf_0$ such that \begin{equation}
  \label{eq:BandA}
  \1_B=Y=\1_A\,, \quad \tilde \PP_\pi\as
  \end{equation}
 Eq.~\eqref{eq:BandA} and the invariance of $A$ then shows that
 $$
 \tilde \PP_\pi(\1_B \circ \tilde \shift=\1_A \circ \tilde \shift=\1_A=\1_B)=1\eqsp.
 $$

Now, note that $\mcf_0=\sigma(p)$ where $p$ is defined in \eqref{eq:defiF}. Then, since $B \in \mcf_0$, there exists $C \in \Xsigma^{\otimes \nset}$ such that $B=p^{-1}(C)$ and thus,
\begin{align*}
1&=\tilde \PP_\pi(\1_B=\1_B\circ \tilde \shift)=\tilde \PP_\pi(\1\{p(\cdot) \in C\}=\1\{p\circ \tilde \shift(\cdot) \in C\})\\&=\tilde \PP_\pi(\1_C \circ p=\1_C\circ p \circ \tilde \shift)\\
&\stackrel{(i)}{=}\tilde \PP_\pi(\1_C \circ p=\1_C\circ \shift \circ p)=\tilde \PP_\pi \circ p^{-1}(\1_C =\1_C\circ \shift)\stackrel{(ii)}{=}\PP_\pi(\1_C =\1_C\circ \shift)\eqsp,
\end{align*}
where $\stackrel{(i)}{=}$ follows from $p \circ \tilde \shift=\shift \circ p$ and $\stackrel{(ii)}{=}$ from $\PP_\pi=\tilde \PP_\pi \circ p^{-1}$ (see \eqref{eq:relationTildePetP}). The dynamical system $(\Xset^\nset,\Xsigma^{\otimes \nset}, \PP_\pi,\shift)$ being ergodic, it implies that $\PP_\pi(C) =0$ or $1$ which concludes the proof since
$$
\PP_\pi(C)=\tilde \PP_\pi \circ p^{-1}(C)=\tilde \PP_\pi(B)=\tilde \PP_\pi(A)\eqsp.
$$
\end{proof}
\begin{proposition}\label{prop:ergo-unique}
Let  $(\Xset^\zset,\Xsigma^{\otimes \zset}, \PP,\shift)$  be a measure-preserving dynamical system. Then, the following statements are equivalent:
    \begin{enumerate}[(a)]
    \item \label{item:cns-ergo-general} $(\Xset^\zset,\Xsigma^{\otimes \zset}, \PP,\shift)$  is ergodic.
    \item \label{item:cns-birkhoff} for all measurable function $h: \Xset^{\zset} \to \rset$ satisfying $\PE(h_+)<\infty$,
\begin{equation}\label{eq:birkhoff-dyn}
 n^{-1} \sum_{k=0}^{n-1} h \circ \shift^k \existLim{n \to \infty} \PE(h),\quad \PP\as
\end{equation}
    \end{enumerate}
\end{proposition}
\begin{proof}
We first show that \eqref{item:cns-ergo-general} implies \eqref{item:cns-birkhoff}.
Assume that $\PE(h_+)<\infty$. If $\PE(h_-)<\infty$, then, \eqref{eq:birkhoff-dyn} follows from Birkhoff's ergodic theorem. If $\PE(h_-)=\infty$, then $\PE(h)=-\infty$. Moreover, since for all nonnegative real number $M$,
$$
-M \leq h\1\{h>-M\} \leq h_+ \eqsp,
$$
the monotone convergence theorem applied to the nondecreasing and nonnegative function, $h_+ - h\1\{h>-M\}$ yields
$$
\lim_{M \to \infty} \PE(h\1\{h>-M\}) =\PE(\lim_{M \to \infty}h\1\{h>-M\})=\PE(h)=-\infty, \quad \PP\as
$$
so that $\PE(h\1\{h>-M\})\existLim{M \to \infty} -\infty$. The proof follows from
\begin{multline*}
\limsup_{n \to \infty} n^{-1} \sum_{k=0}^{n-1} h \circ \shift^k \\
\leq \limsup_{n \to \infty} n^{-1} \sum_{k=0}^{n-1} h \circ \shift^k \1(h \circ \shift^k>-M) =\PE(h\1\{h>-M\})\eqsp,
\end{multline*}
by letting $M$ goes to infinity. Conversely, assume \eqref{item:cns-birkhoff}. Let $A \in \Xsigma^{\otimes \zset}$ such that $\1_A \circ \shift=\1_A$. Then,
$$
 n^{-1} \sum_{k=0}^{n-1} \1_A \circ \shift^k \existLim{n \to \infty} \PP(A),\quad \PP\as
$$
which implies, since $\PP\as$,  $\1_A \circ \shift^k=\1_A$,
$$
\1_A=\PP(A),\quad \PP\as
$$
Since $\1_A$ takes value in $\{0,1\}$, then necessarily $\PP(A)=0$ or $1$. The proof is concluded.
\end{proof}

\section{Consistency of Max-estimators using stationary approximations}
\label{sec:consist-max-estim}
Let $\Xset$ be a Polish space equipped with its Borel sigma-field $\Xsigma$ and let $\shift$  the shift operator as defined in \eqref{eq:def-shift}. Assume that $(\Xset^{\zset}, \Xsigma^{\otimes \zset},\PP, \shift )$ is a measure-preserving ergodic dynamical system. Denote by $\PE$ the expectation operator associated to $\PP$.

Let $(\lkdStat{}[\theta]\, , \, \theta \in \Theta)$ be a family of measurable functions $\lkdStat{}[\theta]: \Xset^{\zset} \rightarrow \rset$, indexed by $\theta \in \Theta$ where $(\Theta,\mathsf{d})$ is a compact metric space and denote $\lkdMStat[n]{}[\theta]\eqdef n^{-1} \sum_{k=0}^{n-1} \lkdStat{}[\theta] \circ \shift^k$. Moreover, consider $(\lkdM[n]{}[\theta]\, , \, n \in \nset^*, \, \theta \in \Theta)$ a family of upper-semicontinuous functions $\lkdM[n]{}[\theta]: \Xset^{\zset} \rightarrow \rset$ indexed by $n \in \nset^*$ and $\theta \in \Theta$.
Consider the following assumptions:
\begin{hyp}{C} \label{assum:momentStat}
$\PE\left(\sup_{\theta \in \Theta} \lkdStat[+]{}[\theta]\right)<\infty$,
\end{hyp}

\begin{hyp}{C} \label{assum:continuity}
$\PP\as$, the function $\theta \mapsto \lkdStat{}[\theta]$ is upper-semicontinuous,
\end{hyp}

\begin{hyp}{C} \label{assum:approx-lkd-stat}
$\lim_{n \to \infty} \sup_{\theta \in \Theta} |\lkdM[n]{}[\theta]-\lkdMStat[n]{}[\theta] |=0\, ,\quad \PP\as$
\end{hyp}
Let $\set{\mlStat{n}}{n \in \nset^*} \subset \Theta$ and $\set{\ml{n}}{n \in \nset^*}\subset \Theta$ such that for all $n\geq 1$,
$$
\mlStat{n}\in  \argmax_{\theta \in \Theta} \lkdMStat[n]{}[\theta]\,, \quad \ml{n}\in \argmax_{\theta \in \Theta} \lkdM[n]{}[\theta]\eqsp.
$$
Assumptions \refhyps[C]{assum:momentStat}{assum:continuity} are quite standard and can be adapted directly from \cite{pfanzagl:1969} (which treated the case of independent \sequence{X}[n][\nset]). For the sake of clarity, we provide here a short and self-contained proof.
\begin{theorem} \label{thm:consist-appendix} Assume \refhyps[C]{assum:momentStat}{assum:continuity}.
\begin{enumerate}[(i)]
\item \label{item:consist-stat-ergo}Then,
$\lim_{n \to \infty} \mathsf{d}(\mlStat{n},\Theta_\star)=0,\eqsp \PP\as$
where $\Theta_\star \eqdef \argmax_{\theta \in \Theta} \PE(\lkdStat{}[\theta])$.
\item \label{item:approx-lkd-stat} Assume in addition that \refhyp[C]{assum:approx-lkd-stat} holds. Then, $\lim_{n \to \infty} \mathsf{d}(\ml{n},\Theta_\star) =0\, ,\eqsp \PP\as$ Moreover,
\begin{align} \label{eq:lim-theta-n}
&\lim_{n \to \infty} \lkdM[n]{}[\ml{n}]=\sup_{\theta \in \Theta} \PE(\lkdStat{}[\theta]) \, , \quad \PP\as\\
\forall \theta \in \Theta,\quad &\lim_{n \to \infty} \lkdM[n]{}[\theta]=\PE(\lkdStat{}[\theta]) \, , \quad \PP\as \label{eq:limite-vrais-theta-fixe}
\end{align}
\end{enumerate}
\end{theorem}
\begin{proof}
\begin{subproof}{Proof of \eqref{item:consist-stat-ergo}}
First note that according to  Proposition \ref{prop:ergo-unique} and \refhyp[C]{assum:momentStat}, for all $\theta \in \Theta$, $\lim_{n \to \infty} \lkdMStat[n]{}[\theta]$ exists $\PP\as$, and
\begin{equation}\label{eq:conv-vraisStat-theta}
\lim_{n \to \infty} \lkdMStat[n]{}[\theta]= \lim_{n \to \infty}  n^{-1} \sum_{k=0}^{n-1} \lkdStat{}[\theta] \circ \shift^k=\PE(\lkdStat{}[\theta]),\quad \PP\as
\end{equation}

Let $K$ be a compact subset of $\Theta$. For all $\theta_0 \in K$, $\PP\as$,
\begin{align}
&\limsup_{\rho \to 0} \limsup_{n \to \infty} \sup_{\theta \in \ball{\theta_0}{\rho}} n^{-1} \sum_{k=0}^{n-1}\lkdStat{}[\theta] \circ \shift^k  \nonumber \\
&\quad \leq \limsup_{\rho \to 0} \limsup_{n \to \infty} n^{-1} \sum_{k=0}^{n-1} \sup_{\theta \in \ball{\theta_0}{\rho}} \lkdStat{}[\theta] \circ \shift^k
= \limsup_{\rho \to 0} \PE\left(\sup_{\theta \in \ball{\theta_0}{\rho}} \lkdStat{}[\theta]\right)\eqsp, \label{eq:proof-consist-stat-0}
\end{align}
where the last equality follows from \refhyp[C]{assum:momentStat} and Proposition \ref{prop:ergo-unique}. Moreover, by the monotone convergence theorem applied to the nonincreasing function $\rho\mapsto \sup_{\theta \in \ball{\theta_0}{\rho}} \lkdStat{}[\theta]$, we have
\begin{align}
\limsup_{\rho \to 0} \PE\left(\sup_{\theta \in \ball{\theta_0}{\rho}} \lkdStat{}[\theta]\right) =  \PE\left(\limsup_{\rho \to 0} \sup_{\theta \in \ball{\theta_0}{\rho}} \lkdStat{}[\theta]\right) \leq \PE(\lkdStat{}[\theta_0])\eqsp, \label{eq:proof-consist-stat-1}
\end{align}
where the last inequality follows from \refhyp[C]{assum:continuity}. Combining \eqref{eq:proof-consist-stat-0} and \eqref{eq:proof-consist-stat-1}, we obtain that for all $\eta>0$ and $\theta_0 \in K$, there exists $\rho^{\theta_0}>0$ satisfying
$$
\limsup_{n \to \infty} \sup_{\theta \in \ball{\theta_0}{\rho^{\theta_0}}} n^{-1} \sum_{k=0}^{n-1}\lkdStat{}[\theta] \circ \shift^k \leq \PE(\lkdStat{}[\theta_0])+\eta \leq \sup_{\theta \in K}\PE(\lkdStat{}[\theta])+\eta \, , \quad \PP\as
$$
Since $K$ is a compact subset of $\Theta$, we can extract a finite subcover of $K$ from $\bigcup_{\theta_0 \in K} \ball{\theta_0}{\rho^{\theta_0}}$, so that
\begin{equation}
\limsup_{n \to \infty} \sup_{\theta \in K} n^{-1} \sum_{k=0}^{n-1}\lkdStat{}[\theta] \circ \shift^k \leq \sup_{\theta \in K}\PE(\lkdStat{}[\theta])+\eta \, , \quad \PP\as
\end{equation}
Since $\eta$ is arbitrary, we obtain
\begin{equation}\label{eq:first-ineg-consist-stat}
\limsup_{n \to \infty} \sup_{\theta \in K} n^{-1} \sum_{k=0}^{n-1}\lkdStat{}[\theta] \circ \shift^k \leq \sup_{\theta \in K}\PE(\lkdStat{}[\theta]) \, , \quad \PP\as
\end{equation}
Moreover, $\PP\as$, by \eqref{eq:proof-consist-stat-1}, we get
\begin{equation*}
\limsup_{\rho \to 0} \sup_{\theta \in \ball{\theta_0}{\rho}} \PE\left( \lkdStat{}[\theta]\right)\leq \limsup_{\rho \to 0} \PE\left(\sup_{\theta \in \ball{\theta_0}{\rho}} \lkdStat{}[\theta]\right)\leq \PE(\lkdStat{}[\theta_0])\eqsp,
\end{equation*}
This shows that $\theta \mapsto \PE( \lkdStat{}[\theta])$ is upper-semicontinuous. As a consequence,  $\Theta_\star\eqdef \argmax_{\theta \in \Theta} \PE(\lkdStat{}[\theta])$ is a closed and nonempty subset of $\Theta$ and therefore, for all $\epsilon>0$, $K_\epsilon \eqdef \{\theta \in \Theta; \mathsf{d}(\theta,\Theta_\star)\geq \epsilon\}$ is a compact subset of $\Theta$. Using again the upper-semicontinuity of $\theta \mapsto \PE( \lkdStat{}[\theta])$, there exists $\theta_\epsilon \in K_\epsilon$ such that for all $\thv \in \Theta_\star$,
$$
\sup_{\theta \in K_\epsilon}\PE(\lkdStat{}[\theta])=\PE( \lkdStat{}[\theta_\epsilon])<\PE( \lkdStat{}[\thv])\eqsp.
$$
Finally, combining this inequality with \eqref{eq:first-ineg-consist-stat}, we obtain that $\PP\as$,
\begin{multline}
\limsup_{n \to \infty} \sup_{\theta \in K_\epsilon}\lkdMStat[n]{}[\theta]=\limsup_{n \to \infty} \sup_{\theta \in K_\epsilon} n^{-1} \sum_{k=0}^{n-1}\lkdStat{}[\theta] \circ \shift^k
\leq \sup_{\theta \in K_\epsilon}\PE(\lkdStat{}[\theta])\\
<\PE( \lkdStat{}[\thv])\stackrel{(1)}{=} \lim_{n \to \infty}  \lkdMStat[n]{}[\thv]
\leq \liminf_{n \to \infty}  \lkdMStat[n]{}[\mlStat{n}]\eqsp, \label{eq:second-proof-consist}
\end{multline}
where $(1)$ follows from \eqref{eq:conv-vraisStat-theta}. This inequality ensures that $\mlStat{n} \notin K_\epsilon$ for all $n$ larger to some $\PP\as$ finite integer-valued random variable. This completes the proof of (i)  since $\epsilon$ is arbitrary.
\end{subproof}
\begin{subproof}{Proof of \eqref{item:approx-lkd-stat}}
First note that \eqref{eq:limite-vrais-theta-fixe} follows from \eqref{eq:conv-vraisStat-theta} and \refhyp[C]{assum:approx-lkd-stat}.

Let $\thv$ be any point in $\Theta_\star$. Then, $\PP\as$,
\begin{multline*}
\PE(\lkdStat{}[\thv])\stackrel{(1)}{=}\liminf_{n\to\infty}\lkdMStat[n]{}[\thv] \stackrel{(2)}{\leq} \liminf_{n\to\infty}\lkdMStat[n]{}[\mlStat{n}] \leq \limsup_{n\to\infty}\lkdMStat[n]{}[\mlStat{n}]\\
=\limsup_{n \to \infty} \sup_{\theta \in \Theta}\lkdMStat[n]{}[\theta] \stackrel{(3)}{\leq}  \sup_{\theta \in \Theta}\PE(\lkdStat{}[\theta])=\PE(\lkdStat{}[\thv])\eqsp,
\end{multline*}
where $(1)$ follows from \eqref{eq:conv-vraisStat-theta}, $(2)$ is direct from the definition of $\mlStat{n}$ and $(3)$ is obtained by applying \eqref{eq:first-ineg-consist-stat} with  $K=\Theta$. Thus,
\begin{equation}\label{eq:third-proof-consist}
\lkdMStat[n]{}[\mlStat{n}]\existLim{n\to\infty}\PE(\lkdStat{}[\thv]),\quad \PP\as
\end{equation}
Denote $\delta_n\eqdef \sup_{\theta \in \Theta} |\lkdM[n]{}[\theta]-\lkdMStat[n]{}[\theta] |$. We get
\begin{equation}\label{eq:sandwich}
\lkdMStat[n]{}[\mlStat{n}]-\delta_n \stackrel{(1)}{\leq} \lkdM[n]{}[\mlStat{n}] \stackrel{(2)}{\leq}  \lkdM[n]{}[\ml{n}] \stackrel{(1)}{\leq}  \lkdMStat[n]{}[\ml{n}] +\delta_n \stackrel{(3)}{\leq}  \lkdMStat[n]{}[\mlStat{n}] +\delta_n\eqsp.
\end{equation}
where $(1)$ follows from the definition of $\delta_n$, $(2)$ from the definition of $\ml{n}$ and $(3)$ from the definition of $\mlStat{n}$. Combining the above inequalities with \eqref{eq:third-proof-consist} and \refhyp[C]{assum:approx-lkd-stat} yields \eqref{eq:lim-theta-n}. \eqref{eq:sandwich} also implies that
$$
\lkdMStat[n]{}[\ml{n}]\existLim{n \to \infty}\PE(\lkdStat{}[\thv]),\quad \PP\as
$$
which yields, using \eqref{eq:second-proof-consist},
$$
\limsup_{n \to \infty} \sup_{\theta \in K_\epsilon}\lkdMStat[n]{}[\theta] <\liminf_{n \to \infty}\lkdMStat[n]{}[\ml{n}]=\limsup_{n \to \infty}\lkdMStat[n]{}[\ml{n}]=\PE( \lkdStat{}[\thv])\,,\quad \PP\as
$$
where $K_\epsilon \eqdef \{\theta \in \Theta; \mathsf{d}(\theta,\Theta_\star)\geq \epsilon\}$. Therefore, $\ml{n} \notin K_\epsilon$ for all $n$ larger to some $\PP\as$-finite integer-valued random variable. The proof is completed since $\epsilon$ is arbitrary.
\end{subproof}
\end{proof}

\begin{lemma} \label{lem:useful}
Let \sequence{V}[n][\nset] be a sequence of strict-sense stationary random variables on the same probability space $(\Omega, \mcf, \PP)$. Denote by $\PE$ the associated expectation operator and assume that $\PE{[(\ln |V_0|)_+]}<\infty$. Then, for all $\eta \in (0,1)$,
$$
\lim_{k \to \infty} \eta^k V_k =0\, , \quad \PP\as
$$
\end{lemma}
\begin{proof}
Let $\eta \in (0,1)$. For all $\epsilon>0$,
\begin{align*}
&\sum_{k=1}^\infty\PP(\eta^k |V_k|\geq \epsilon)=\sum_{k=1}^\infty \PP(\ln |V_0|-\ln \epsilon \geq -k\ln \eta )\\
&\quad \leq \sum_{k=1}^\infty \PP\left( \frac{(\ln(V_0))_+-\ln \epsilon}{-\ln \eta}\geq k\right)<\infty\eqsp,
\end{align*}
where the last inequality follows from $\PE{[(\ln |V_0|)_+]}<\infty$. The proof follows by applying the Borel-Cantelli lemma.
\end{proof}

\section*{References}
\bibliographystyle{elsarticle-harv}

\end{document}